\documentclass{amsart}
\usepackage{amsmath}%
\usepackage{amsfonts}%
\usepackage{amssymb}%
\usepackage{graphicx}
\newtheorem{theorem}{Theorem}[section]
\newtheorem{lemma}[theorem]{Lemma}
\newtheorem{corollary}[theorem]{Corollary}
\newtheorem{proposition}[theorem]{Proposition}

\newtheorem{remark}{Remark}
\theoremstyle{definition}
\newtheorem{definition}[theorem]{Definition}
\numberwithin{equation}{section}

\usepackage[utf8]{inputenc}
\usepackage{verbatim}
\usepackage{mathrsfs}
\usepackage{colortbl}
\usepackage{subfigure}
\usepackage{psfrag}
\usepackage[all]{xy}

\newcommand{\N}{\mathbb{N}}
\newcommand{\Z}{\mathbb{Z}}
\newcommand{\R}{\mathbb{R}}
\newcommand{\C}{\mathbb{C}}

\newcommand{\Hess}{\textnormal{Hess}}

\newcommand{\signal}{\textnormal{signal}}
\newcommand{\tatop}[2]{\genfrac{}{}{0pt}{10}{#1}{#2}}
\usepackage{mathdots}%to make the reverse diagonal dots  

\begin{document}

\title[Polynomials of degree $4$ and real Jacobian conjecture]{On polynomial submersions of degree $4$ and the real Jacobian conjecture in $\R^2$}

\author{Francisco Braun}
\address[F. Braun and B. Or\'{e}fice-Okamoto]
{Departamento de Matem\'{a}tica, Universidade Federal de S\~{a}o Carlos \newline 
 Rod. Washington Lu\'{i}s, Km 235 - C.P. 676 - 13565-905 S\~{a}o Carlos, SP - Brasil}%
\author{Bruna Or\'{e}fice-Okamoto}
\email[F. Braun]{franciscobraun@dm.ufscar.br}%
\email[B. Or\'{e}fice-Okamoto]{bruna@dm.ufscar.br}

\thanks{The second author was partially supported by FAPESP grants 2011/08877-3 and 2013/14014-3.}

\thanks{This paper is in final form and no version of it will be submitted for
publication elsewhere.}
\date{May 5, 2014.}
\subjclass[2010]{Primary: 14R15; Secondary: 26C10, 53C12.}
%26C10: Polynomials: location of zeros
%35F05: Linear first-order equations
%35A30: Geometric theory, characteristics, transformations
%53C112: Foliations (differential geometric aspects)

\keywords{Real Jacobian conjecture, global injectivity, positive polynomials, half-Reeb components}

\begin{abstract}
The main result of this paper is the following version of the real Jacobian conjecture: ``Let $F=(p,q):\R^2\to\R^2$ be a polynomial map with nowhere zero Jacobian determinant. If the degree of $p$ is less than or equal to $4$, then $F$ is injective''. 
Assume that two polynomial maps from $\R^2$ to $\R$ are equivalent when they are the same up to affine changes of coordinates in the source and in the target. 
We completely classify the polynomial submersions of degree $4$ with at least one disconnected level set up to this equivalence, obtaining four classes. 
Then, analysing the half-Reeb components of the foliation induced by a representative $p$ of each of these classes, we prove there is not a polynomial $q$ such that the Jacobian determinant of the map $(p,q)$ is nowhere zero. 
Recalling that the real Jacobian conjecture is true for maps $F=(p,q)$ when all the level sets of $p$ are connected, we conclude the proof of the main result.
\end{abstract}

\maketitle

\section{Introduction}

Let $F=(p,q):\R^2\to\R^2$ be a polynomial map such that its Jacobian determinant, $\det DF$, is nowhere zero in $\R^2$. 
By the inverse function theorem, $F$ is locally injective.  
The \emph{real Jacobian conjecture} asserts that $F$ is globally injective. 
This is closely related to the famous \emph{Jacobian conjecture},  which claims that given $K$ a field of characteristic zero, any polynomial map from $K^n$ to $K^n$ such that its Jacobian determinant is equal to $1$  is injective. 
Jacobian conjecture was stated at the first time in 1939 by Keller (\cite{Kel}) and up to now is open if $n\geq 2$. 
We refer to \cite{Ess} for further informations on Jacobian conjecture.

The real Jacobian conjecture is nevertheless not true: in 1994, Pinchuk constructed, in \cite{Pi}, a non injective polynomial map $F=(p,q)$ such that $\det DF\neq 0$ in $\R^2$. 
In this example, the polynomials $p$ and $q$ have high  degrees: $p$ has degree $10$ and $q$ has degree $40$. 
Since the injectivity of $F$ is clear when the degree  of $F$ is one (in this case $F$ is an affine change of coordinates), it is natural to ask what would be the highest degree of $p$ or $q$ guaranteeing the global injectivity of $F$. 

In this direction,  Gwo\'{z}dziewicz proved, in  \cite{Gw}, that if the degrees of $p$ and $q$ are less than or equal to $3$ then $F$ is injective. In \cite{BS}, Braun and Santos generalized  Gwo\'{z}dziewicz result proving that the injectivity of $F$ is true provided just the degree of $p$ being less than or equal to $3$. 
Thus a natural question is: what is the maximum degree of $p$, between $3$ and $9$, in order that  the injectivity of $F$ is necessary independently of the degree of $q$? As a partial answer to this question, we prove in this note the following result. 

\begin{theorem}\label{main}
Let $F=(p,q):\R^2\to\R^2$ be a polynomial map such that $\det DF\neq 0$ in $\R^2$. If the degree of $p$ is less than or equal to $4$, then $F$ is injective. 
\end{theorem}

Since the assumption ``$\det DF\neq 0$ in $\R^2$'' gives that $q$ is strictly monotone along each connected component of a level set of $p$ (as the level sets of $q$ are transversal to the level sets of $p$), one strategy to show the injectivity of $F$ is to prove that the level sets of $p$ are all connected. 
Actually, since injective polynomial maps are bijections (see \cite{BBR}), this is also necessary to the injectivity of $F$.
Thus to prove Theorem \ref{main} it is enough to consider all the polynomial submersions $p$ with  disconnected level sets and to show that for these    there is not a polynomial $q$ such that $\det DF\neq 0$ in $\R^2$. Before continuing, we need the following  definition.

\begin{definition}\label{def:equi}
We say that two functions $p,q:\R^2\to\R$ are \emph{equivalent} if there exist an affine change of coordinates\footnote{By affine change of coordinates we mean a map of the form $\left(\begin{array}{c}
x\\
y
\end{array}\right)\mapsto\left(\begin{array}{cc}
a & b\\
c & d
\end{array}\right)\left(\begin{array}{c}
x\\
y
\end{array}\right)+\left(\begin{array}{c}
e\\
f
\end{array}\right)$, where $a,b,c,d,e,f\in\R$, with $ad-bc\neq0$.} 
$T:\R^2\to\R^2$ and constants $M,N\in\R$, with $M\neq 0$, such that $p(x,y)=Mq\circ T^{-1}(x,y)+N$.
\end{definition}

Let $p$ be a polynomial submersion. It is quite simple to show that  all its  level sets are connected if the degree of $p$ is less than or equal to $2$. 
If $p$ has degree $3$ and has at least one disconnected level set, one of the results of the above cited paper \cite{BS} shows that $p$  is equivalent to $y+xy^2$. 
In our next theorem, we give the classification in the case of degree $4$.

\begin{theorem}\label{Classification} 
If $p:\R^2\to \R$ is a polynomial submersion of degree $4$ which has at least one disconnected level set, then $p$ is equivalent to one of the following
\begin{enumerate}
\item\label{1} $p(x,y)=y+xy^2+y^4$,
\item\label{2} $p(x,y)=y+a_{02}y^2+xy^3$, with $a_{02}=0$ or $1$, 
\item\label{3} $p(x,y)=y+x^2y^2$,
\item\label{4} $p(x,y)=y+a_{02}y^2+y^3+x^2y^2$, with $a_{02}^2-3<0$.
\end{enumerate}
\end{theorem}

For each case of Theorem \ref{Classification}, we will prove there is no  polynomial $q$ such that $\det D(p,q)$ is nowhere zero. 
As a consequence, we will obtain Theorem \ref{main}, since when the degree of $p$ is less than or equal to $3$, the injectivity was already proven in \cite{BS} (in that  paper, it was also shown there is no $q$ such that $\det D(y+xy^2,q)$ is nowhere zero).

The paper is organized as follows. In Section \ref{section_common_zeros}, we use subresultants of two polynomials to develop results to decide when two special polynomials in $\R[x,y]$ have common zeros. 
These special polynomials have the form of $\frac{\partial p}{\partial x}$ and $\frac{\partial p}{\partial y}$, where $p$ is a polynomials of degree $4$.
In Section \ref{section_level_sets}, we construct results to decide when a special polynomial submersion has all its level sets connected. 
Then in Section \ref{section_classification}, we use the results of sections \ref{section_common_zeros} and \ref{section_level_sets} to prove Theorem \ref{Classification}.

In Section \ref{quase_final}, we will consider the polynomials of Theorem \ref{Classification} and prove that for each of them there is not a polynomial $q(x,y)$ such that $\det D(p,q)(x,y)>0$, for all $(x,y)\in\R^2$. 
For the arguments we shall use the concept of half-Reeb component of a foliation. Then, in Section \ref{fim}, we depicted the proof of Theorem \ref{main}. 

The arguments in Section \ref{quase_final} are divided in two groups. In the first one, we study the polynomials \ref{1} and \ref{2} of Theorem \ref{Classification} and use techniques analogous to \cite{BS} to conclude the non existence of a polynomial $q$.
Now in the second group, when we analyse the polynomial \ref{3} and \ref{4}, these techniques no longer work. Thus we transform part of the problem in being able to decide when a special polynomial of one variable is positive. 
Namely, Lemma \ref{bruna} asserts that  $L(\theta)=\sum_{j=0}^{N}b_j\big(2(j+1)\theta +2j+1\big)\theta^j$ can not be a positive polynomial, if not identically zero. The proof of this lemma is in the Appendix.

\section{Common zeros of polynomials}\label{section_common_zeros}
Let us recall the concept of subresultants of polynomials.  Consider two polynomials $p(x),q(x)\in \C[x]$, 
\begin{align*}
p(x) & = a_nx^n+a_{n-1}x^{n-1}+\cdots+a_0, \\ 
q(x) & = b_mx^m+b_{m-1}x^{m-1}+\cdots+b_0,
\end{align*}
and take their Sylvester matrix written in the following form
$$
Syl(p,q,x)=\left(\begin{array}{cccccccc}
a_n    & a_{n-1} & \cdots     & \cdots  & a_0     &        &         &        \\
       & a_n     & a_{n-1}    & \cdots  & \cdots  & a_0    &         &        \\
       &         & \ddots     & \ddots  &         &        & \ddots  &        \\
       &         &            & a_n     & a_{n-1} & \cdots & \cdots  & a_0    \\
\cdots & \cdots  &            & \cdots  & \cdots  & \cdots & \cdots  & \cdots \\
       &         & \iddots    & \iddots &         &        & \iddots &        \\
       & b_m     & b_{m-1}    & \cdots  & \cdots  & b_0    &         &        \\
b_m    & b_{m-1} & \cdots     & \cdots  & b_0     &        &         &          
\end{array}\right)\hspace{-1cm}
\begin{array}{c}
\left.\begin{array}{c}\phantom{a_0}\\ \phantom{a_0}\\ \phantom{\ddots}\\ \phantom{a_0}
\end{array}\right\} m \textnormal{ rows} \\
\left.\begin{array}{c}\phantom{\cdots}\\ \phantom{\iddots}\\ \phantom{a_0}\\ \phantom{a_0}
\end{array}\right\} n \textnormal{ rows}
\end{array}.
$$
For each $k\in\left\{0,1,\ldots,[(n+m)/2]\right\}$, we define the \emph{$k$-subresultant} of $p$ and $q$, $R_k(p,q,x)$, by the determinant of the $(n+m-2k)\times (n+m-2k)$ matrix obtained 
when we delete the first and the latest $k$ columns and rows of $Syl(p,q,x)$.     The following result is classical but we include a proof of it for the sake of completeness. Our proof was based on \cite{Bo}.
\begin{lemma}\label{resultant}
Consider $p(x)$ and $q(x)$ as above, with $a_nb_m\neq 0$. Then $p$ and $q$ have exactly $k$ common roots (counting multiplicity) if and only if 
$$
R_0(p,q,x)=\cdots=R_{k-1}(p,q,x)=0,\ \ \ \ \ \ \ \ \ R_k(p,q,x)\neq 0.
$$  
\end{lemma}
\begin{proof}
Let us denote by $C_j$ the $j$th column of $Syl(p,q,x)$, for $j=1,\ldots,m+n$. By substituting the last column $C_{m+n}$ by $C_{m+n}+\sum_{k=1}^{m+n-1}y^{m+n-k}C_k$, and calculating the determinant using the Laplace expansion on the last column, we get that
$R_0(p,q,x)=f(y)p(y)+g(y)q(y)$, where $f$ and $g$ are polynomials of degree less than or equal to $m-1$ and $n-1$, respectively. Since $R_0(p,q,x)$ is a complex number with does not depend on $y$, it follows that  $p$ and $q$ have a common root if and only if $R_0(p,q,x)=0$.

Now the lemma follows readily from the following assertion: \emph{If $p(x)=(x-\alpha)p_1(x)$ and $q(x)=(x-\alpha)q_1(x)$, $\alpha\in\C$, 
then $R_i(p_1,q_1,x)=R_{i+1}(p,q,x)$, for $i=0,1,\ldots$}. To prove the assertion, we make the following operations on the columns of $Syl(p_1,q_1,x)$: we change column $C_i$ by $C_i-\alpha C_{i-1}$ for $i=m+n-2,m+n-3,\ldots,2$.
Then we observe this is exactly the matrix  $Syl(p,q,x)$ without the first and the last rows and columns.     
\end{proof}

\begin{lemma}\label{Calculo_res}
Let $a,b,c,d,e:\R\to\R$ be functions. For 
$$
p(x,y)=y^2+a(x)y+b(x),\ \ \ \ \ \ q(x,y)=y^2+c(x)y+d(x),
$$
we have the following formulas
\begin{align*}
R_1(p,q,y)=     &c-a,\\
R_0(p,q,y)=     &-(d-b)^2+\big(a(d-b)-bR_1\big)R_1.
\end{align*}
Whereas for 
$$
p(x,y)=y^2+a(x)y+b(x),\ \ \ \ \ \ q(x,y)=y^3+c(x)y^2+d(x)y+e(x),
$$
we have 
\begin{align*}
R_1(p,q,y)=     &a^2-ac-b+d,\\
R_0(p,q,y)=     &-(ab-bc+e)^2+\big(a(ab-bc+e)-bR_1\big)R_1.
\end{align*}
%Here we are writing $a=a(x)$, $b=b(x)$, $c=c(x)$, $d=d(x)$, $e=e(x)$, $R_0=R_0(p,q,y)$ and $R_1=R_1(p,q,y)$. 
\end{lemma}
\begin{proof}
The proof is straightforward.
\end{proof}

\begin{lemma}\label{DerivadaPar}
Let $f,g,h,A,B:\R\to\R$ be smooth functions, and $\alpha\in\R$ be such that $h(\alpha)\neq 0$ and $A(\alpha)^2-4B(\alpha)<0$.
Define 
$$
R(x)=h(x)\Big(-f(x)^2+\big(A(x)f(x)-B(x)g(x)\big)g(x)\Big).
$$
If there is a non negative integer $i$ such that  $R^{(0)}(\alpha)=R^{(1)}(\alpha)=\cdots={R}^{(2i)}(\alpha)=0$, then ${R}^{(2i+1)}(\alpha)=0$.\footnote{Here $R^{(j)}$ stands for the $j$th derivative of $R$.}  
\end{lemma}
\begin{proof}
In order to simplify the notation, we denote $A=A(\alpha)$, $B=B(\alpha)$, $f^{(j)}=f^{(j)}(\alpha)$, $g^{(j)}=g^{(j)}(\alpha)$ and $h^{(j)}=h^{(j)}(\alpha)$.

We will first make the proof supposing $h\equiv 1$. If $R(\alpha)=0$, we have  
$$
0=-f^2+\left(Af-Bg\right)g=-\left(f^2-Agf+Bg^2\right),
$$
which only can be true if $\left(Ag\right)^2-4Bg^2\geq0$. Since by hypothesis $A^2-4B<0$, we obtain $g=0$ and thus $f=0$. 
This implies by Leibniz rule that 
$$
{R}'(\alpha)=-2ff'+\left(Af-Bg\right)'g+\left(Af-Bg\right)g'=0,
$$
showing the result for $i=0$. 
Let us make  the following induction hypothesis:
\begin{equation}\label{ddd}
\textnormal{The result is true for }i-1, \textnormal{ and for each } k\leq i-1,\ f^{(k)}=g^{(k)}=0.
\end{equation}
We will show this is true for $i$. Then the first part of the proof will be completed. Leibniz rule gives
$$
R^{(2i)}(\alpha)=-\sum_{k=0}^{2i}\binom{2i}{k}\left(f^{(k)}f^{(2i-k)}-\left(Af-Bg\right)^{(k)}g^{(2i-k)}\right).
$$
By \eqref{ddd}, if $k\leq i-1$, $f^{(k)}=g^{(k)}=0$, hence $(Af-Bg)^{(k)}=0$ and $(Af-Bg)^{(i)}=Af^{(i)}-Bg^{(i)}$ (apply Leibniz rule). Now if $k\geq i+1$, $f^{(2i-k)}=g^{(2i-k)}=0$, since  $2i-k\leq i-1$. Then last sum simplifies to 
$$
{R}^{(2i)}(\alpha)  = -\binom{2i}{i}\left((f^{(i)})^2-Ag^{(i)}f^{(i)}+B(g^{(i)})^2\right).
$$
If ${R}^{(2i)}(\alpha)=0$, we have to have $(Ag^{(i)})^2-4B(g^{(i)})^2\geq0$,
which by hypothesis gives $g^{(i)}=0$, hence $f^{(i)}=0$. 
Now it remains to prove that ${R}^{(2i+1)}(\alpha)=0$. Apply again Leibniz rule  to obtain
$$
{R}^{(2i+1)}(\alpha)=-\sum_{k=0}^{2i+1}\binom{2i+1}{k}\left(f^{(k)}f^{(2i+1-k)}-\left(Af-Bg\right)^{(k)}g^{(2i+1-k)}\right).
$$
If $k\leq i$, $f^{(k)}=0$ and $\left(Af-Bg\right)^{(k)}=0$.
If $k>i$, $2i+1-k<i+1$, and it follows that $f^{(2i+1-k)}=g^{(2i+1-k)}=0$. This gives us ${R}^{(2i+1)}(\alpha)=0$.

In the general case,  we define $\overline{R}(x)=R(x)/h(x)$. Leibniz rule gives for $l=0,1,2\ldots$ 
$$
\overline{R}^{(l)}=\sum_{k=0}^{l}\binom{l}{k}R^{(k)}\left(1/h\right)^{(l-k)}.
$$
Then the hypothesis $R^{(0)}(\alpha)=\cdots=R^{(2i)}(\alpha)=0$ shows $\overline{R}^{(0)}(\alpha)=\cdots=\overline{R}^{(2i)}(\alpha)=0$. Thus by the first part of the proof we obtain $\overline{R}^{(2i+1)}(\alpha)=0$. Hence, putting $l=2i+1$ above, we conclude $R^{(2i+1)}(\alpha)=0$.
\end{proof}

Now we apply the preceding results to produce criteria to decide when two special polynomials have common zeros.  

\begin{theorem}\label{raiz_23}
Let $p(x,y)$ and $q(x,y)$ in $\R[x,y]$ as follows 
$$
p(x,y)=M(x)y^2+a(x)y+b(x),\ \ \ \ \ q(x,y)=N(x)y^3+c(x)y^2+d(x)y+e(x).
$$
Write $R_0(x)=R_0(p,q,y)=b_0+b_1x+\cdots+b_lx^l$. If
\begin{enumerate}
\item $R_0(x)$ has no common zeros with $N(x)$,
\item\label{raiz_23_2} there exists $z\in\R$ such that $b_lR_0(z)<0$,
\end{enumerate}
then there exists $(\alpha,\beta)\in\R^2$ such that $p(\alpha,\beta)=q(\alpha,\beta)=0$.
\end{theorem}
\begin{proof}
Let $\alpha_1,\ldots,\alpha_k$ be the distinct real zeros of $R_0(x)$ (by  \ref{raiz_23_2} we have at least one). 

If $M(\alpha_i)=0$, for some $i\in\{1,\ldots,k\}$, we have two possibilities: $$
a(\alpha_i)\neq 0\ \ \ \ \ \textnormal{or}\ \ \ \ \ a(\alpha_i)=0.
$$
In the first one, define $\overline{p}(y)=p(\alpha_i,y)$ and $\overline{q}(y)=q(\alpha_i,y)$, and observe that $R_0(\overline{p},\overline{q},y)=R_0(\alpha_i)/N(\alpha_i)=0$, and $R_1(\overline{p},\overline{q},y)=a(\alpha_i)^2$.
Thus by Lemma \ref{resultant} there is exactly one $\beta\in\R$ such that $p(\alpha_i,\beta)=q(\alpha_i,\beta)=0$ (since the coefficients of $\overline{p}$ and $\overline{q}$ are real).
On the other hand, if $a(\alpha_i)=0$, it is simple to see that $R_0(\alpha_i)=N(\alpha_i)^2b(\alpha_i)^3$. 
Thus $b(\alpha_i)=0$, hence $p(\alpha_i,y)\equiv 0$. Moreover, there exists $\beta\in\R$ such that $q(\alpha_i,\beta)=0$, since $q(\alpha_i,y)$ is a polynomial of degree $3$ in $y$.

From now on, we will suppose $M(\alpha_i)\neq 0$, for all $i\in\{1,\ldots,k\}$. 
Denoting $R_1(x)=R_1(p,q,y)$, if $R_1(\alpha_i)\neq0$ for some $i\in\{1,\ldots,k\}$, Lemma \ref{resultant} shows the existence of $\beta\in\R$ such that $p(\alpha_i,\beta)=q(\alpha_i,\beta)=0$ as we wanted. 

Thus we suppose $R_1(\alpha_i)=0$ for each $i\in\{1,\ldots,k\}$.
\emph{We assert there is $i\in\{1,\ldots,k\}$ such that $a(\alpha_i)^2-4M(\alpha_i)b(\alpha_i)\geq0$}. 
In this case, since by Lemma \ref{resultant} $p(\alpha_i,y)$ and $q(\alpha_i,y)$ have exactly two zeros in common (observe  $R_2(p,q,y)(\alpha_i)=M(\alpha_i)\neq 0$), they must be real zeros and we are done.

Let us then prove the assertion. Suppose by contradiction that for all $i\in\{1,\ldots,k\}$, $a(\alpha_i)^2-4M(\alpha_i)b(\alpha_i)<0$. 
By definition of subresultant and by Lemma \ref{Calculo_res}, we have in a neighborhood of each $\alpha_i$ where $M(x)N(x)\neq 0$, 
\begin{equation}\label{raiz_23_equa}
R_0(p,q,y)=M^3N^2R_0\left(p/M,q/N,y\right)=M^3N^2\left(-f^2+(Af-Bg)g\right),
\end{equation}
where we have written $A=a/M$, $B=b/M$, $f=ab/M^2-bc/(MN)+e/N$ and $g=R_1(p/M,q/N,y)$.  Then by Lemma \ref{DerivadaPar}, since $a^2-4Mb<0$ if and only if $A^2-4B<0$, we obtain that each $\alpha_i$ is a zero of multiplicity even of $R_0(x)$, say $m_i$. Then
$$
R_0(x)=b_l(x-\alpha_1)^{m_1}\cdots(x-\alpha_k)^{m_k}r(x),
$$
where $r(x)$ is a monic polynomial without real zeros. 
Thus $b_lR_0(x)\geq 0$, for all $x\in\R$, a contradiction with assumption  \ref{raiz_23_2}.
\end{proof}

\begin{theorem}\label{raiz_22}
Let $p(x,y)$ and $q(x,y)$ in $\R[x,y]$ be defined by 
$$
p(x,y)=M(x)y^2+a(x)y+b(x),\ \ \ \ \ q(x,y)=N(x)y^2+c(x)y+d(x).
$$
Write $R_0(x)=R_0(p,q,y)=b_0+b_1x+\cdots+b_lx^l$.
If 
\begin{enumerate}
\item\label{raiz_22_1} $R_0(x)$ has no common zeros with $M(x)$ and $N(x)$,
\item\label{raiz_22_2} there exists $z\in\R$ such that $b_lR_0(z)<0$,
\end{enumerate}
then there exists $(\alpha,\beta)\in\R^2$ such that $p(\alpha,\beta)=q(\alpha,\beta)=0$.

Moreover, if we keep other hypotheses and change \ref{raiz_22_1} by
\begin{itemize}
\item[(i$'$)]\label{78} $R_0(x)$ has no common zeros with $N(x)$ and there is $\alpha\in\R$ with $R_0(\alpha)=M(\alpha)=0$,
\end{itemize}
then there exists $\beta\in\R$ such that $p(\alpha,\beta)=q(\alpha,\beta)=0$ if and only if $a(\alpha)\neq 0$ or $p(\alpha,y)\equiv 0$ and $c(\alpha)^2-4N(\alpha)d(\alpha)\geq 0$.
\end{theorem}

\begin{proof}
As in the proof of Theorem \ref{raiz_23}, let $\alpha_1,\ldots,\alpha_k$ be the distinct real zeros of $R_0(x)$. 

If $M(\alpha_i)=0$ for some $i$, and $a(\alpha_i)\neq 0$, take $\overline{p}(y)=p(\alpha_i,y)$ and $\overline{q}(y)=q(\alpha_i,y)$, and observe that  $R_0(\overline{p},\overline{q},y)=R_0(\alpha_i)/N(\alpha_i)=0$. 
Since $R_1(\overline{p},\overline{q},y)=a(\alpha_i)$, we have by Lemma \ref{resultant} that there is $\beta\in\R$ such that $p(\alpha_i,\beta)=q(\alpha_i,\beta)=0$. 
If now $a(\alpha_i)=0$, $R_0(\alpha_i)=N(\alpha_i)^2b(\alpha_i)^2$, hence   $b(\alpha_i)=0$. Thus $p(\alpha_i,y)\equiv0$, and  $q(\alpha_i,y)=0$ for some $y\in\R$ if and only if $c(\alpha_i)^2-4N(\alpha_i)d(\alpha_i)\geq 0$.
This proves the second part of the theorem.

We suppose now that $M(\alpha_i)\neq 0$ for each $i\in\{1,\ldots,k\}$. The proof from now on is similar to the proof of Theorem \ref{raiz_23}. 
If $R_1(\alpha_i)\neq0$ for some $i$, we are done. If $R_1(\alpha_i)=0$ for each $i=1,\ldots,k$, \emph{we assert there is $i$ such that $a(\alpha_i)^2-4M(\alpha_i)b(\alpha_i)\geq 0$}. 
If this is in force, since $p(\alpha_i,y)$ and $q(\alpha_i,y)$ have two zeros in common, they must be real. 

Thus let us prove the assertion. Suppose $a(\alpha_i)^2-4M(\alpha_i)b(\alpha_i)< 0$, for each $i\in\{1,\ldots,k\}$. 
By Lemma \ref{Calculo_res} and definition of resultant, we get in a small neighbourhood of each $\alpha_i$
\begin{equation}\label{raiz_22_equa}
R_0(p,q,y)=M^2N^2\left(-f^2+(Af-Bg)g\right),
\end{equation}
where $A=a/M$, $B=b/M$, $f=d/N-b/M$ and $g=R_1(p/M,q/N,y)$. Then by Lemma \ref{DerivadaPar}, we have that all the zeros of $R_0(x)$ have multiplicity even $m_i$. Thus 
$$
R_0(x)=b_l(x-\alpha_1)^{m_1}\cdots (x-\alpha_k)^{m_k}r(x),
$$
where $r(x)$ is a monic polynomial without real zeros. This gives $b_lR_0(x)\geq0$ for all $x\in\R$, a contradiction with hypothesis \ref{raiz_22_2}. 
\end{proof}

The following two corollaries use the $1$-subresultant to analyse common zeros of polynomials when hypotheses \ref{raiz_23_2} of the preceding  theorems are difficult to be verified.

\begin{corollary}\label{raiz_23_Cor}
Let $p(x,y)$ and $q(x,y)$ as in Theorem \ref{raiz_23}. If
\begin{enumerate}
\item $R_0(x)$ has no common zeros with $N(x)$,
\item\label{iiiii} there exists $z\in\R$ such that $R_1(z)=0$, $b_lM(z)>0$ and $N(z)\neq 0$,
\end{enumerate}
then there exists $(\alpha,\beta)\in\R^2$ such that $p(\alpha,\beta)=q(\alpha,\beta)=0$.
\end{corollary}

\begin{proof}
Since $M(z)N(z)\neq 0$, in a neighborhood of $z$ we have, as in  \eqref{raiz_23_equa},
\begin{equation}\label{789}
R_0(x)=R_0(p,q,y)=N^2M^3\left(-f^2+(Af-Bg)g\right),
\end{equation}
If $f(z)\neq 0$, then $b_lR_0(z)=-N(z)^2M(z)^2f(z)^2b_lM(z)<0$ (recall that $g=R_1(p/M,q/N,y)$), and thus the result follows from Theorem \ref{raiz_23}.

If, on the other hand, $f(z)=0$, we have $R_0(z)=R_1(z)=0$, which guarantees two common zeros of $p(z,y)$ and $q(z,y)$. 
If these zeros are real we are done. 

Thus let us suppose $a(z)^2-4M(z)b(z)<0$. If $g\equiv 0$, since the zeros of $f$ are isolated (it is a rational function not identically zero, since if $f\equiv 0$, $R_0\equiv 0$ by \eqref{789}, hence $b_l=0$, a contradiction with \ref{iiiii}), we have by \eqref{789} for $x\neq z$ near $z$ 
$$
b_lR_0(x)=-N(x)^2M(x)^2f(x)^2b_lM(x)<0,
$$
and we are under the hypotheses of Theorem \ref{raiz_23}. Now if $g\not\equiv 0$, we have $g(x)\neq 0$ and $a(x)^2-4M(x)b(x)<0$ for $x\neq z$ near $z$. 
This gives $\left(A(x)g(x)\right)^2-4B(x)g(x)^2<0$ which guarantees  $-f(x)^2+A(x)g(x)f(x)-B(x)g(x)^2<0$ for $x\neq z$ near $z$. 
This together with  \eqref{789} and assumption \eqref{iiiii} gives 
$$
b_lR_0(x)=N(x)^2M(x)^2b_lM(x)\left(-f(x)^2+A(x)g(x)f(x)-B(x)g(x)^2\right)<0,
$$
for $x\neq z$ near $z$, and we are again  under the hypotheses of Theorem \ref{raiz_23}.
\end{proof}

\begin{corollary}\label{raiz_22_Cor}
Let $p(x,y)$ and $q(x,y)$ as in Theorem \ref{raiz_22}. If 
\begin{enumerate}
\item\label{7897} $R_0(x)$ has no common zeros with $M(x)$ and $N(x)$,
\item\label{7898} $b_l>0$ and there exists $z\in\R$ such that $R_1(z)=0$ and $M(z)N(z)\neq 0$,
\end{enumerate}
then there exists $(\alpha,\beta)\in\R^2$ such that $p(\alpha,\beta)=q(\alpha,\beta)=0$.

Moreover, if we keep  other hypotheses and change \ref{raiz_22_1} by
\begin{itemize}
\item[(i$'$)] $R_0(x)$ has no common zeros with $N(x)$ and there is $\alpha\in\R$ with $R_0(\alpha)=M(\alpha)=0$,
\end{itemize}
then there exists $\beta\in\R$ such that $p(\alpha,\beta)=q(\alpha,\beta)=0$ if and only if $a(\alpha)\neq 0$ or $p(\alpha,y)\equiv 0$ and $c(\alpha)^2-4N(\alpha)d(\alpha)\geq 0$.
\end{corollary}
\begin{proof}
The second part of the corollary is clearly similar to the second part of Theorem \ref{raiz_22}. Thus suppose we are under hypothesis \ref{7897} and \ref{7898}.

Since $M(z)N(z)\neq 0$, in a neighborhood of $z$ we have, as in equation \eqref{raiz_22_equa},
\begin{equation}\label{799}
R_0(p,q,y)=N^2M^2\left(-f^2+(Af-Bg)g\right).
\end{equation}
If $f(z)\neq 0$, we have $b_lR_0(z)=-N(z)^2M(z)^2f(z)^2b_l<0$, hence we are under the hypotheses of Theorem \ref{raiz_22}.

If, on the other hand, $f(z)=0$, we have $R_0(z)=R_1(z)=0$, which guarantees two common zeros of $p(z,y)$ and $q(z,y)$. If these zeros are real we are done.
Thus let us suppose $a(z)^2-4M(z)b(z)<0$. If $g\equiv 0$, since the zeros of $f$ are isolated (as in the proof of Corollary \ref{raiz_23_Cor}), we  have $b_lR_0(x)=-N(x)^2M(x)^2f(x)^2b_l<0$ for $x\neq z$ near $z$, and we are again  under the hypotheses of Theorem \ref{raiz_22}. 
Now if $g\not\equiv 0$, we have $g(x)\neq 0$ and $a(x)^2-4M(x)b(x)<0$ for $x\neq z$ near $z$. 
This gives $\left(A(x)g(x)\right)^2-4B(x)g(x)^2<0$ and thus $-f(x)^2+A(x)g(x)f(x)-B(x)g(x)^2<0$ for $x\neq z$ near $z$. 
This together with  \eqref{799} and assumption \ref{7898} gives 
$$
b_lR_0(x)=N(x)^2M(x)^2b_l\left(-f(x)^2+A(x)g(x)f(x)-B(x)g(x)^2\right)<0,
$$
for $x\ne z$ near $z$, and we are also under hypotheses of Theorem \ref{raiz_22}.
\end{proof}

\section{Level sets}\label{section_level_sets}
 
\begin{lemma}\label{unbounded}
Let $M\subset \R^2$ be an open set. If $p:M\to\R$ is a smooth submersion then the connected components of the level sets of $p$ induce a smooth foliation of dimension $1$ of $M$. 
\end{lemma}
\begin{proof}
See, for example, \cite{Cam}.
\end{proof}
In particular, when $M=\R^2$, any connected component of a level set of $f$ is an unbounded curve in both directions.

\begin{lemma}\label{zero_interval}
Let $p(x,y)=a_n(y)x^n+\cdots+a_0(y)$ be a smooth function. Then for each  interval $[c,d]\subset\{y\in\R\ |\ a_n(y)\neq 0\}$, there exists an  interval $[a,b]$ such that 
$$\{(x,y)\in\R^2\ |\ y\in [c,d],\ p(x,y)=0\}\subset [a,b]\times [c,d].$$
\end{lemma}
\begin{proof}
Define 
$$
A_i=\sup_{y\in[c,d]}\left|\frac{a_i(y)}{a_n(y)}\right|,\ i=0,1,\ldots,n-1,\ \ \ \ \ \ \ \ \ \ \  B=1+\sum_{k=1}^n\left(nA_{n-k}\right)^{1/k}.
$$
If $y\in[c,d]$ and $|x|>B$, then 
$$
p(x,y)=a_n(y)x^n\left(1+\frac{a_{n-1}(y)}{a_n(y)}\frac{1}{x}+\cdots+\frac{a_0(y)}{a_n(y)}\frac{1}{x^n}\right)\neq 0.
$$
Taking $[a,b]=[-B,B]$, the lemma follows.
\end{proof}

\begin{proposition}\label{Level_Sets_2ttt}
Let $p:\R^2\to\R$ be a smooth submersion with the expression
$$
p(x,y)=A(y)x^2+B(y)x+C(y).
$$
If
\begin{enumerate}
\item $A(y)\neq0$, $\forall y\in\R$,
\item $\Delta(y)=B(y)^2-4A(y)C(y)$ is a \emph{polynomial} with \emph{odd} degree,
\end{enumerate}
then $p^{-1}\{0\}$ is connected.
\end{proposition}
\begin{proof}
We  suppose that the leader coefficient of $\Delta(y)$ is positive. The proof in the other case is analogous. 
Therefore there exists $c\in\R$ such that $\Delta(c)=0$ and $\Delta(y)<0$, $\forall y< c$. 
In particular,  $p^{-1}\{0\}\subset \R\times[c,\infty)$ and there exists exactly one $x_c\in\R$ such that  $p(x_c,c)=0$. 
Let $\Gamma$ be the connected component of $p^{-1}\{0\}$ which contains $(x_c,c)$ and consider $Q=(q_1,q_2)\in\R\times(c,\infty)$ such that $p(Q)=0$. It is clearly enough to prove that $Q\in\Gamma$.

By Lemma \ref{zero_interval}, there exists an  interval $[a,b]$ such that 
\begin{equation}\label{imagem_inversa_limitada}
p^{-1}\{0\}\cap\left(\R\times[c,q_2]\right)\subset(a,b)\times[c,q_2].
\end{equation}
Now by Lemma \ref{unbounded}, both ends of $\Gamma$ must escape the compact $[a,b]\times[c,q_2]$. 
Therefore, by \eqref{imagem_inversa_limitada}, $\Gamma$ will cut the line $\R\times\{q_2\}$ in two points, say $(x_1,q_2)$ and $(x_2,q_2)$. 
Since $\Gamma$ does not have self intersections, $x_1\neq x_2$. 
In particular $\Delta(q_2)>0$ and $q_1\in\{x_1,x_2\}$, hence $Q\in\Gamma$.    
\end{proof}

\begin{proposition}\label{Level_Sets_22tt}
Let $p:\R^2\to\R$ be the smooth submersion 
$$
p(x,y)=A(y)x^2+B(y)x+C(y).
$$
If
\begin{enumerate}
\item $A(y)=0$ for exactly one $y\in\R$,
\item\label{segundo} When $A(y)=0$, there exists exactly one $x\in\R$ such that $p(x,y)=0$,
\item\label{terceiro} $\Delta(y)=B(y)^2-4A(y)C(y)$ is a \emph{polynomial} with \emph{even} degree and its leader coefficient is \emph{negative},
\end{enumerate}
then $p^{-1}\{0\}$ is connected.
\end{proposition}

\begin{proof}
Let $y_1\in\R$ such that $A(y_1)=0$. By item \ref{terceiro}, there exist $y_0,y_3\in\R$, with $y_0<y_1<y_3$, such that if $y\leq y_0$ or if $y\geq y_3$, $\Delta(y)<0$. 
Hence 
\begin{equation}\label{duh}
p^{-1}\{0\}\subset\R\times(y_0,y_3).
\end{equation}
By item \ref{segundo}, there exists exactly one point $P$ in the line $\R\times\{y_1\}$ such that $p(P)=0$. 
Let $\Gamma$ be the connected component of $p^{-1}\{0\}$ which contains $P$. 

\emph{We assert that $\Gamma$ can not be entirely contained in $\R\times(y_0,y_1]$ nor in $\R\times[y_1,y_3)$}. 
Indeed, if $\Gamma\subset\R\times[y_1,y_3)$, then there exists $y_2>y_1$ such that (using lemmas \ref{zero_interval} and \ref{unbounded}) there exist four different solutions of $p(x,y_2)=0$, a contradiction since this is a quadratic equation in $x$. 
Similar contradiction can be obtained if we suppose $\Gamma\subset(y_0,y_1]$.

Now take $Q=(\overline{x},\overline{y})\in\R\times[y_1,y_3]$ such that $p(Q)=0$. We will prove that $Q\in\Gamma$. 
In fact, by the assertion above, there exists $(x_2,y_2)\in\Gamma$ with $y_1<y_2<\overline{y}$. 
Thus by Lemma \ref{zero_interval}, there exists an interval $[a,b]$ such that $p^{-1}\{0\}\cap\left(\R\times[y_2,y_3]\right)\subset[a,b]\times[y_2,y_3]$.
Then by \eqref{duh} and Lemma \ref{unbounded}, the connected component of $p^{-1}\{0\}$ which contains $Q$ must cut the line $\R\times\{y_2\}$ in two different points. 
Since $p(x,y_2)=0$ is a quadratic equation in $x$, one of these points must be $(x_2,y_2)$, and then $Q\in\Gamma$. 

The same can be done if we take $Q\in\R\times[y_0,y_1]$ such that $p(Q)=0$. Thus $p^{-1}\{0\}$ is connected. 
\end{proof}

To prove propositions \ref{Level_Sets_2ttt} and \ref{Level_Sets_22tt}, we used the  properties of the discriminant of a quadratic equation.
For the next proposition, which already appeared in \cite{BS}, let us recall the properties of the discriminant of a cubic equation
$$
x^3+Ax^2+Bx+C=0.
$$
Take $P=B-A^2/3$ and $Q=C-AB/3+2A^3/27$, and define the \emph{discriminant}  of the above equation by $D=Q^2/4 + P^3/27$. We have that  
\begin{align*}
\textrm{if } D<0, & \textrm{ the equation has three distinct, real solutions,}\\
\textrm{if } D=0, & \textrm{ the equation has three real solutions, with two being equal,}\\
\textrm{if } D>0, & \textrm{ the equation has one real and two complex solutions.}
\end{align*}
As a consequence, we obtain 
\begin{proposition}\label{Level_Sets_3}
Let $p:\R^2\to\R$ be the smooth submersion 
$$
p(x,y)=x^3+A(y)x^2+B(y)x+C(y).
$$
If the discriminant $D(y)$ of the equation $p(x,y)=0$ is a \emph{polynomial} with \emph{even} degree and with \emph{positive} leader coefficient, then $p^{-1}\{0\}$ is connected.
\end{proposition}
\begin{proof}
The proof is similar to the proof of propositions \ref{Level_Sets_2ttt},  and \ref{Level_Sets_22tt}, thus we give just the outline. 
Take $[c,d]$ such that $D(y)>0$, $\forall y\notin[c,d]$. 
Prove first that $p^{-1}\{0\}\cap\big(\R\times[d,\infty)\big)$ and $p^{-1}\{0\}\cap\big(\R\times(-\infty,c]\big)$ are non empty connected sets. 
Then conclude these two sets connect  each other by a curve in $[a,b]\times[c,d]$, where $[a,b]$ is given by Lemma \ref{zero_interval}.  

\end{proof}

\section{The proof of Theorem \ref{Classification}}\label{section_classification}
In order to prove Theorem \ref{Classification}, we shall use next corollary, which is an immediate consequence of the classification of homogeneous polynomials of degree $4$ given in Theorem 2.6 of \cite{CL}.

\begin{corollary}\label{second_classif}
Any polynomial of degree exactly $4$ can be transformed by means of a \emph{linear change of variables} in one of the following.
\begin{align*}
& (I) && p_3+x^4+6\mu x^2y^2+y^4,  && \mu<-{1}/{3},\\
& (II) && p_3+\alpha\left(x^4+6\mu x^2y^2+y^4\right),  && \alpha=\pm 1,\ \mu>-1/3,\ \mu\neq 1/3,\\
& (III) && p_3+x^4+6\mu x^2y^2-y^4, && \mu\in\R, \\
& (IV)&& p_3+\alpha y^2\left(6x^2+y^2\right), && \alpha=\pm 1,\\
& (V) && p_3+\alpha y^2\left(6x^2-y^2\right), && \alpha=\pm 1,\\
& (VI) && p_3+\alpha\left(x^2+y^2\right)^2, && \alpha=\pm 1,\\
& (VII) && p_3+6\alpha x^2y^2, && \alpha=\pm 1,\\
& (VIII) && p_3+4x^3y, && \\
& (IX) && p_3+\alpha x^4, && \alpha=\pm 1,
\end{align*}
where $p_3$ is a polynomial of degree less than or equal to $3$.
\end{corollary}
We will work with the polynomials of Corollary \ref{second_classif} showing that each of them is not a submersion, or has all its level sets connected, or is equivalent (in the sense of Definition \ref{def:equi}) to one of the cases of Theorem \ref{Classification}. We divide the 9 cases in 5 groups, each of them in one of the subsections bellow. 
Theorem \ref{Classification} will be a direct consequence of the propositions \ref{I_III_VIII_IX}, \ref{Pro_case_V}, \ref{Pro_case_VII}, \ref{Pro_Case_IV} and \ref{Pro_case_II_VI} contained in the subsections \ref{primeira}, ..., \ref{ultima}, respectively.

The arguments use the results of sections \ref{section_common_zeros} and \ref{section_level_sets}. The calculations of the subresultants and of the discriminants  were made using Maple.

To stablish notation, we  write the polynomial $p_3$ in the following form
$$
p_3=a_{10}x+a_{01}y+a_{20}x^2+a_{11}xy+a_{02}y^2+a_{30}x^3+a_{21}x^2y+a_{12}xy^2+a_{03}y^3.
$$

\subsection{Cases (I), (II), (III)  and (VI)}\label{primeira}

\begin{proposition}\label{I_III_VIII_IX}
The polynomials of cases (I), (II), (III) and (VI) are not submersions.
\end{proposition}
Next lemma will help us in the proof of Proposition \ref{I_III_VIII_IX}.
\begin{lemma}\label{lemma_I_III_VIII_IX}
The polynomials of cases (I), (II), (III) and (VI) are equivalent to
$$
p(x,y)=p_3(x,y)+x^4+\gamma x^2y^2+\theta y^4,\  \textnormal{ with } \theta=\pm 1.
$$
Moreover, if $\theta=1$, then $\gamma\neq -2$.
\end{lemma}

\begin{remark}
We remark that if $\theta=1$ and $\gamma=-2$, then the change of coordinates $T(x,y)=(x-y,x+y)$ transforms $p(x,y)$ in case (VII).
\end{remark}

\begin{proof}
Polynomial (I) is counted above taking $\gamma=6\mu$ and $\theta=1$. Dividing polynomial (II) by $\alpha$ and taking $\gamma=6\mu$, we have that it is   also counted in the lemma, with $\theta=1$. 
Now case (III) is counted above with $\gamma=6\mu$ and $\theta=-1$. 
Finally, dividing polynomial (VI) by $\alpha$, it is in the form of the lemma by taking $\gamma=2$ and $\theta=1$.
\end{proof}
\begin{proof}[Proof of Proposition \ref{I_III_VIII_IX}]
We will analyse the polynomial of Lemma \ref{lemma_I_III_VIII_IX}.

Observe  $\frac{\partial p}{\partial y}$ and $\frac{\partial p}{\partial x}$ have the forms $M(y)x^2+a(y)x+b(y)$ and $N(y)x^3+c(y)x^2+d(y)x+e(y)$, respectively, with $N(y)\equiv 4$. 
With the notations of Theorem \ref{raiz_23}, observe 
$$
R_0\left(\frac{\partial p}{\partial y},\frac{\partial p}{\partial x},x\right)=b_0+b_1y+\cdots+b_8y^8+b_9y^9,
$$
with $b_9=-64\theta(4\theta-\gamma^2)^2$. If $b_9\neq0$, this theorem gives that there exists a common zero of  $\frac{\partial p}{\partial y}$ and $\frac{\partial p}{\partial x}$. 
Therefore $p$ is not a submersion. 
We suppose thus  $b_9=0$, i.e. $\theta=1$ and $\gamma=2$. In this case we get $b_8=0$ and $b_7=-576\left((a_{12}-a_{30})^2+(a_{21}-a_{03})^2\right)$. Then by Theorem \ref{raiz_23}, if $a_{12}\neq a_{30}$ or $a_{21}\neq a_{03}$, $p$ is not a submersion.

On the other hand, if $a_{30}=a_{12}$ and $a_{03}=a_{21}$, we have $b_6=0$ and $b_5=-16\big((a_{12}^2+4a_{02}-4a_{20}-a_{21}^2)^2+4(a_{12}a_{21}-2a_{11})^2\big)$. 
If $b_5\neq 0$, the same theorem guarantees that $p$ is not a submersion.

Now if $b_5=0$, i.e. $a_{02}=-a_{12}^2/4+a_{20}+a_{21}^2/4$ and $a_{11}=a_{12}a_{21}/2$, then  $b_4=0$ and $b_3=-\big((8a_{10}+a_{12}^3-4a_{20}a_{12})^2+
(a_{12}^2a_{21}-4a_{21}a_{20}+8a_{01})^2\big)$. Then, as above, if $b_3\neq 0$, $p$ is not a submersion, whereas if $a_{10}=-a_{12}^3/8+a_{20}a_{12}/2$ and $a_{01}=-a_{12}^2a_{21}/8+a_{20}a_{21}/2$, we get $\frac{\partial p}{\partial x}(x,y)=-(a_{12}+4x)\big(a_{12}^2-4a_{12}x-8x^2-4a_{20}-8y^2-4a_{21}y\big)/8$ and $\frac{\partial p}{\partial y}(x,y)=-(a_{21}+4y)\big(a_{12}^2-4a_{12}x-8x^2-4a_{20}-8y^2-4a_{21}y\big)/8$, which clearly have a common zero and so $p$ is not a submersion. 
\end{proof}

\subsection{Case (VIII)}
 
\begin{proposition}\label{Pro_case_V}
If a polynomial of case (VIII) is a submersion, it is equivalent to 
$$
p=y+a_{02}y^2+xy^3,\ \ a_{02}=0\textnormal{ or } 1.
$$
In particular, the level set $p=0$ is not connected.
\end{proposition}

To prove this proposition, we will first prove the following lemma.
\begin{lemma}\label{V}
The polynomials of case (VIII) are equivalent to one of the following.
\begin{align}
&p= a_{10}x+a_{01}y+a_{20}x^2+a_{11}xy+a_{02}y^2+     a_{12}xy^2+a_{03}y^3+x^3y,\ \label{V_1}\\
& a_{03}^2+a_{12}^2+a_{02}^2>0,\notag\\
&p= a_{10}x+a_{01}y+a_{20}x^2+a_{11}xy+ x^3y,\ \ \ \ \ 4a_{11}^3+27a_{01}^2\neq 0 \label{V_3}\\
&p= a_{10}x+     2y+a_{20}x^2     -3xy+ x^3y,\label{V_4}\\   
&p= a_{10}x+        a_{20}x^2+                                                    x^3y,\ \ \ \ \ a_{20}=0\textnormal{ or } 1,\label{V_5}
\end{align}
\end{lemma}
\begin{proof}
By multiplying the polynomial in case (VIII) by $1/4$, and after composing it with the change of  coordinates $T(x,y)=\big(x+a_{21}/3,y+a_{30}\big)$, we obtain it is equivalent to (keeping the notations of the coefficients)
$$
p=a_{10}x+a_{01}y+a_{20}x^2+a_{11}xy+a_{02}y^2+a_{12}xy^2+a_{03}y^3+x^3y.
$$
If $a_{03}^2+a_{12}^2+a_{02}^2>0$, we obtain case \eqref{V_1}, whereas if these coefficients are zero, we could have $a_{01}^2+a_{11}^2$ zero or not.
In the first case we get \eqref{V_5} (if $a_{20}\neq0$, simply divide $p$ by $a_{20}$ and compose it with the change of coordinates $(x,y)\mapsto \big(x,y/a_{20}\big)$) and in the second case, if $4a_{11}^3+27a_{01}^2\neq 0$, we get case \eqref{V_3}. 
On the other hand, if $4a_{11}^3+27a_{01}^2= 0$, then $a_{01}\neq 0$ and the change  $T(x,y)=\big(\sqrt[3]{2/a_{01}}x,a_{01}y/2\big)$ gives case \eqref{V_4}.   
\end{proof}

\begin{proof}[Proof of Proposition \ref{Pro_case_V}]
For case \eqref{V_1} and \eqref{V_3} of Lemma \ref{V}, $\frac{\partial p}{\partial x}$ and $\frac{\partial p}{\partial y}$ have the forms $M(y)x^2+a(y)x+b(y)$ and $N(y)x^3+c(y)x^2+d(y)x+e(y)$, respectively, with $M(y)=3y$ and $N(y)= 1$. Thus we are  under the hypotheses of Theorem \ref{raiz_23}. Calculating the resultant we have 
$$
R_0\left(\frac{\partial p}{\partial x},\frac{\partial p}{\partial y},x\right)=b_0+b_1y+\cdots+b_7y^7.
$$
In case \eqref{V_1}, $b_7=-243a_{03}^2$. Thus if $a_{03}\neq0$, $p$ is not a submersion. On the other hand, if $a_{03}=0$, we have ($b_7=0$) and $b_6=-25a_{12}^3$. Then observe that
$$
R_1\left(\frac{\partial p}{\partial x},\frac{\partial p}{\partial y},x\right)=4a_{20}^2-3a_{10}y+6a_{11}y^2+15a_{12}y^3.
$$
Therefore if $a_{12}\neq 0$ and $a_{20}\neq0$, we have a zero $z$ of $R_1(y)$ with opposite signal with $a_{12}$. 
Then $b_6M(z)=-75a_{12}^3z>0$, and Corollary \ref{raiz_23_Cor} guarantees that $p$ is not a submersion. 

On the other hand, if $a_{12}\neq0$ and $a_{20}=0$, we have that $b_0=-a_{10}^3$. 
Suppose first $a_{10}\neq 0$. If $a_{10}$ and $a_{12}$ have opposite signals, then $b_0$ and $b_6$ have opposite signals and Theorem \ref{raiz_23} gives that $p$ is not a submersion. 
Now if $a_{10}$ and $a_{12}$ have the same signal, we observe $R_1(y)$ have a negative and a positive zero, which gives $b_6M(z)>0=-75a_{12}^3z>0$, for one of these zeros. 
This proves that $p$ is not a submersion by Corollary \ref{raiz_23_Cor}. 
Now if $a_{10}=0$, we observe that $\frac{\partial p}{\partial x}(x,0)\equiv 0$, and $\frac{\partial p}{\partial y}(x,0)$ is a polynomial of degree $3$. Thus, again, $p$ is not a submersion.

If now $a_{12}=0$, then $a_{02}\neq0$ and hence $b_5=-108a_{02}\neq0$. Then Theorem \ref{raiz_23} gives that $p$ is not a submersion.

In the case \eqref{V_3}, $b_7=b_6=b_5=b_4=0$ and $b_3=-\big(4a_{11}^3+27a_{01}^2\big)\neq0$, which by Theorem \ref{raiz_23} proves that $p$ is not a submersion. 

Now for the case \eqref{V_4}, it is simple to observe that $\frac{\partial p}{\partial y}(-2,y)\equiv 0$, and $\frac{\partial p}{\partial x}(-2,y)$ is a polynomial of degree $1$, hence $p$ is not a submersion.

Finally, in the case \eqref{V_5}, it is clear that if $a_{10}=0$ then $p$ is not a submersion, whereas if $a_{10}\neq 0$, we multiply $p$ by $1/a_{10}^2$ and apply the change $(x,y)\mapsto \big(a_{10}y,x/a_{10}\big)$ to get the  polynomial map of the proposition. It is clear this is a submersion. 
\end{proof}

\subsection{Case (IX)}

\begin{proposition}\label{Pro_case_VII}
If the polynomial of the case (IX) is a submersion and has at least one disconnected  level set,  then it is equivalent to 
$$
p=y+xy^2+y^4.
$$
\end{proposition}

\begin{lemma}\label{VII}
The polynomials in case (IX) are equivalent to one of the following.
\begin{align}
&a_{10}x+a_{01}y+a_{20}x^2\hspace{-.1cm}+a_{11}xy+a_{02}y^2\hspace{-.1cm}+a_{30}x^3\hspace{-.1cm}+a_{21}x^2y+
a_{12}xy^2\hspace{-.1cm}+y^3\hspace{-.1cm}+x^4,\label{VII_1}\\
&a_{10}x+a_{01}y+a_{20}x^2\hspace{-.1cm}+a_{11}xy+a_{02}y^2\hspace{-.1cm}+a_{30}x^3\hspace{-.1cm}
+a_{21}x^2y+a_{12}xy^2\hspace{-.1cm}+    x^4,\label{VII_2}\\
& a_{12}\neq 0,\ \textnormal{or}\ a_{12}=a_{21}=0, \notag \\
&a_{10}x+        a_{20}x^2+         a_{02}y^2+a_{30}x^3+      x^2y+               x^4,\ \ a_{02}\neq 0,\label{VII_3}\\
&a_{10}x+a_{01}y+a_{20}x^2+  a_{30}x^3+  x^2y+  x^4,\ \ a_{01}\neq 0,\label{VII_4}\\
&      a_{10}x+                                                     x^2y+               x^4,\label{VII_5}
\end{align}
\end{lemma}
\begin{proof}
First we divide the polynomial of case (IX) by $\alpha$. Then we get the case \eqref{VII_1} above if $a_{03}\neq0$ by composing $p$ with $T(x,y)=\big(x,\sqrt[3]{a_{03}}y\big)$. 
If $a_{03}=0$ and $a_{12}\neq0$, we get the first part of case \eqref{VII_2} above, whereas if $a_{03}=a_{12}=a_{21}=0$, we obtain the second one. 
Now if $a_{03}=a_{12}=0$ and $a_{21}\neq0$, we take the transformation $T(x,y)=\big(x+a_{11}/(2a_{21}),a_{21}y\big)$ to obtain
$$
p=a_{10}x+a_{01}y+a_{20}x^2+a_{02}y^2+a_{30}x^3+x^2y+x^4.
$$
Then, if $a_{02}\neq0$, we apply  $T(x,y)=\big(x,y+a_{01}/(2a_{02})\big)$ to obtain case \eqref{VII_3}. 
On the other hand, if $a_{02}=0$ we obtain case \eqref{VII_4} if $a_{01}\neq 0$ and, if $a_{01}=0$, the transformation $T(x,y)=\big(x, y+a_{30}x+a_{20}\big)$ gives case \eqref{VII_5}. 
\end{proof}

\begin{proof}[Proof of Proposition \ref{Pro_case_VII}]
The equation $p(x,y)-c=0$, where $p$ is the polynomial of case \eqref{VII_1} of Lemma \ref{VII}, has the form 
$$
y^3+A(x)y^2+B(x)y+C(x)=0,
$$
where
$A(x)=a_{02}+a_{12}x$, $B(x)=a_{01}+a_{11}x+a_{21}x^2$ and $C(x)=a_{10}x+a_{20}x^2+a_{30}x^3+x^4-c$. 
Calculating the discriminant of the equation as in Proposition \ref{Level_Sets_3}, we obtain 
$$
D(x)=\sum_{i=0}^kb_ix^i,\ \ \ \ \ \  k=8,\ \ \ \ \ \ b_8=1/4.
$$ 
Thus by this proposition, we conclude that if $p$ is a submersion, it has all its level sets connected.

Now we observe that in case \eqref{VII_2} of Lemma \ref{VII},  
\begin{equation}\label{special_form}
\frac{\partial p}{\partial y}=M(y)x^2+a(y)x+b(y),\ \ \ \ \frac{\partial p}{\partial x}=N(y)x^3+c(y)x^2+d(y)x+e(y),
\end{equation}
with $N(y)=4$, and $M(y)=a_{21}$ (which \emph{can} be zero, for the firs part of this case  and \emph{is} zero for the second part). 
Then we calculate $R_0\big(\frac{\partial p}{\partial y}, \frac{\partial p}{\partial x},x\big)=b_0+b_1y+\cdots+b_5y^5$, and observe that $b_5\neq0$ if $a_{12}\neq 0$, which by Theorem \ref{raiz_23} gives
\footnote{Observe when $M\equiv0$, we can yet apply this theorem.} 
that $p$ is not a submersion. Now if $a_{12}=a_{21}=0$, we have that $b_5=b_4=0$ and $b_3=-128 a_{02}^3$.
Therefore, if $a_{02}\neq 0$, the same theorem guarantees $p$ is not a submersion. 
Finally, if $a_{02}=0$, we have $b_{2}=0$ and $b_1=4a_{11}^4$. Then if $a_{11}\neq 0$, we also have $p$ is not a submersion. 
If $a_{11}=0$ and $a_{01}\neq 0$, it is simple to see the level sets of $p$ are all connected, while if $a_{11}=a_{01}=0$, it follows that  $\frac{\partial p}{\partial y}(x,y)\equiv 0$, and $\frac{\partial p}{\partial x}(x,y)$ annihilates for some $x$, since it has degree $3$ in $x$. 

Consider now case \eqref{VII_3} of Lemma \ref{VII} and observe that $\frac{\partial p}{\partial y}$ and $\frac{\partial p}{\partial x}$ have again the form of \eqref{special_form}. 
We have $R_0\big(\frac{\partial p}{\partial y}, \frac{\partial p}{\partial x},x\big)=b_0+b_1y+b_2y^2+b_3y^3$, with $b_3=-8a_{02}(4a_{02}-1)^2$. If this is not zero, we obtain by Theorem \ref{raiz_23} that $p$ is not a submersion. 
If this is zero, i.e. $a_{02}=1/4$ (since $a_{02}\neq 0$), we obtain 
\begin{equation}\label{r0VII_3}
R_0\left(\frac{\partial p}{\partial y},\frac{\partial p}{\partial x},x\right)=-\frac{9}{4}a_{30}^2y^2 +(3a_{30}a_{10}-2a_{20}^2)y-a_{10}^2. 
\end{equation}
Moreover, equation $p(x,y)-c=0$ has the form of Proposition  \ref{Level_Sets_2ttt}, with discriminant 
\begin{equation}\label{deltaVII_3}
\Delta(x)=-a_{30}x^3-a_{20}x^2-a_{10}x+c.
\end{equation}
Thus if $a_{30}\neq0$, Proposition \ref{Level_Sets_2ttt} gives us that if $p$ is a submersion, then all its level sets are connected. 
On the other hand, supposing $a_{30}=0$, if $a_{20}\neq 0$, \eqref{r0VII_3} and Theorem \ref{raiz_23}  give us $p$ is not a submersion. 
If $a_{20}=0$ and $a_{10}\neq 0$, \eqref{deltaVII_3} and Proposition \ref{Level_Sets_2ttt} give us again that if $p$ is a submersion, all its level sets are connected. 
Finally, if $a_{20}=a_{10}=0$, it is clear that $\nabla p(0,0)=0$.

Now we consider case \eqref{VII_4}. We have $\frac{\partial p}{\partial x}=a_{10}+2a_{20}x+3a{30}x^2+4x^3+2xy$ and  $\frac{\partial p}{\partial y}=a_{01}+x^2$.
If $a_{01}<0$, it is clear that $p$ is not a submersion. 
On the other hand, if $a_{01}>0$, we have that  $p$ is a submersion, but in this case $p(x,y)=c$ gives us $y=\big(c-a_{10}x
-a_{20}x^2-a_{30}x^3-x^4\big)/\left(a_{01}+x^2\right)$, i.e. all the level sets of $p$ are connected.

Finally, case \eqref{VII_5} with $a_{10}=0$ is not a submersion, whereas if $a_{10}\neq0$, we divide $p$ by $\sqrt[3]{a_{10}^4}$ and then the change  $(x,y)\mapsto \big(y/\sqrt[3]{a_{10}^2},x/\sqrt[3]{a_{10}} \big)$, gives the polynomial of the proposition. 
\end{proof}

\subsection{Case (VII)}

\begin{proposition}\label{Pro_Case_IV}
If the polynomial of case (VII) is a submersion and has no all its level sets connected, it is equivalent to one of the following.
\begin{align*}
p=& y+a_{02}y^2+y^3+x^2y^2,\ \ \ \ a_{02}^2-3<0,\\
p=& y+x^2y^2.
\end{align*} 
\end{proposition}

\begin{lemma}\label{IV}
The polynomials of case (VII) are equivalent to one of the following.
\begin{align}
p=& a_{10}x+a_{01}y+a_{20}x^2+a_{11}xy+a_{02}y^2+    x^3+    y^3+x^2y^2,\label{IV_1}\\
p=& a_{10}x+a_{01}y+a_{20}x^2+a_{11}xy+a_{02}y^2+            y^3+x^2y^2,\label{IV_2}\\
p=& a_{10}x+a_{01}y+a_{20}x^2+a_{11}xy+a_{02}y^2+                x^2y^2,\ \ \ 
|a_{02}|=1,\label{IV_3}\\
p=& a_{10}x+a_{01}y+           a_{11}xy+                          x^2y^2.\label{IV_4}
\end{align} 
\end{lemma}
\begin{proof}
We first divide $p$ by $6\alpha$. Then we compose it with the transformation $T(x,y)=\big(x+a_{12}/2,y+a_{21}/2\big)$ to get the following form
$$
p=a_{10}x+a_{01}y+a_{20}x^2+a_{11}xy+a_{02}y^2+a_{30}x^3+a_{03}y^3+x^2y^2.
$$
If $a_{03}a_{30}\neq 0$, we compose the polynomial  with the transformation  $T(x,y)= \big(x/\sqrt[3]{a_{30}a_{03}^2},y/\sqrt[3]{a_{30}^2a_{03}}\big)$ and multiply by $a_{30}^{-2}a_{03}^{-2}$ to get \eqref{IV_1}. 
If $a_{30}=0$ and $a_{03}\neq 0$, the transformation $T(x,y)=\big(x/\sqrt[3]{a_{03}},\sqrt[3]{a_{03}}y\big)$ takes $p$ to \eqref{IV_2}. 
If $a_{30}\neq 0$ and $a_{03}=0$, we change $x$ by $y$ to get the case just studied. 

If now $a_{30}=a_{03}=0$, and if $a_{02}\neq 0$, the change of coordinates $T(x,y)=\left(1/\sqrt{|a_{02}|}x,\sqrt{|a_{02}|}y\right)$ gives \eqref{IV_3}. If $a_{20}\neq 0$, change $x$ by $y$ to get the case just studied.

Thus we suppose $a_{30}=a_{03}=a_{20}=a_{02}=0$, to obtain case \eqref{IV_4}. 
\end{proof}

\begin{proof}[Proof of Proposition \ref{Pro_Case_IV}]
We will analyse each case of Lemma \ref{IV}. For case \eqref{IV_1}, we notice that $\frac{\partial p}{\partial x}$ and $\frac{\partial p}{\partial y}$ have the forms $M(x)y^2+a(x)y+b(x)$ and $N(x)y^2+c(x)y+d(x)$, respectively, with $M(x)=2x$ and $N(x)=3$. 
The subresultant is
$$
R_0\left(\frac{\partial p}{\partial x},\frac{\partial p}{\partial y},y\right)=b_0+\cdots+b_7x^7,
$$
with $b_7=-24$ and $b_0=-3a_{11}^2a_{01}+6a_{11}a_{10}a_{02}-9a_{10}^2$. Thus  by Theorem \ref{raiz_22}, $\nabla p(x,y)$ has a zero if $b_0\neq 0$.
Moreover, by the same theorem, if $b_0=0$ and $a_{11}=a(0)\neq 0$, we have again that $\nabla p(x,y)$ has a zero. 
Now if $b_0=0$ and $a_{11}=0$, we obtain $a_{10}=0$ and then  $\frac{\partial p}{\partial x}(x,y)=x\big(2a_{20}+3x+2y^2\big)$. 
Taking  $x=-2\big(y^2+a_{20}\big)/3$, we have $\frac{\partial p}{\partial x}=0$, and  $\frac{\partial p}{\partial y}$ is a polynomial of degree $5$ in $y$. Hence $p$ is not a submersion.

Now consider case \eqref{IV_2} and observe that the equation $p(x,y)-c=0$ is of the form $y^3+A(x)y^2+B(x)y+C(x)$, with the discriminant $\Delta(x)$ being  a polynomial of degree $8$ with leader coefficient $a_{20}/27$. 
If $a_{20}>0$, Proposition \ref{Level_Sets_3} gives us that if $p$ is a submersion, all its level sets are connected.

If $a_{20}\leq 0$, observe that  $\frac{\partial p}{\partial x}$ and $\frac{\partial p}{\partial y}$ have the forms $M(x)y^2+a(x)y+b(x)$ and $N(x)y^2+c(x)y+d(x)$, respectively, with $M(x)=2x$ and $N(x)=3$. Then we have
\begin{equation}\label{101920}
R_0\left(\frac{\partial p}{\partial x},\frac{\partial p}{\partial y},y\right)=b_0+\cdots+b_5x^5+b_6x^6,
\end{equation}
with $b_6=-16a_{20}$ and $b_0=-3a_{11}^2a_{01}+6a_{11}a_{10}a_{02}-9a_{10}^2$, and  
$$
R_1\left(\frac{\partial p}{\partial x},\frac{\partial p}{\partial y},y\right)=-3a_{11}+4a_{02}x+4x^3.
$$
We first suppose $a_{20}<0$. Then if $b_0\neq 0$ and $a_{11}\neq 0$, we have by Corollary \ref{raiz_22_Cor}, that $\nabla p(x,y)$ has a zero. 
If $b_0\neq 0$ and $a_{11}=0$, we get $b_0=-9a_{10}^2<0$ and $b_6>0$, which by Theorem \ref{raiz_22} gives a zero of $\nabla p$. 
Now if $b_0=0$ and $a_{11}=a(0)\neq 0$, we have by the same theorem, that $p$ is not a submersion, whereas if $a_{11}=0$, then $a_{10}=0$. 
Therefore, again by Theorem \ref{raiz_22}, if $c(0)^2-4N(0)d(0)=4\big(a_{02}^2-3a_{01}\big)\geq0$, we have a zero of $\nabla p$. 
Thus let us suppose $a_{02}^2-3a_{01}<0$. 
We observe that $y=-\sqrt{-a_{20}}$ annihilates $\frac{\partial p}{\partial x}(x,y)$, and there exists $x$ such that $\frac{\partial p}{\partial y}(x,y)=2yx^2+\big(3y^2+2a_{02}y+a_{01}\big)=0$, by our last hypothesis.

Now if $a_{20}=0$, then $b_6=0$ and $b_5=-8a_{10}$. Thus if $a_{10}\neq 0$ and $b_0\neq 0$, Theorem \ref{raiz_22} gives a zero of $\nabla p$. If yet $a_{10}\neq 0$ and $b_0=0$, the same theorem guarantees $p$ is not a submersion, since in this case $a(0)=a_{11}$ has to be non zero. 
Now if $a_{10}=0$, we see that $b_5=b_4=0$ and $b_3=-4a_{11}a_{01}$ and $b_0=-3a_{11}^2a_{01}$. 
Therefore, if $a_{01}a_{11}\neq 0$, Theorem \ref{raiz_22} guarantees that $p$ is not a submersion. 
If $a_{01}=0$, it is clear that $\nabla p(0,0)=0$, whereas if $a_{01}\neq0$ and $a_{11}=0$, we obtain that $p$ is a submersion if and only if $a_{02}^2-3a_{01}<0$. 
Then $a_{01}>0$ and multiplying $p$ by $1/\sqrt{a_{01}^3}$ and composing with the change $(x,y)\mapsto\big(x/\sqrt[4]{a_{01}},y/\sqrt{a_{01}} \big)$, we get the first case of this  proposition.

For case \eqref{IV_3}, $\frac{\partial p}{\partial y}=0$ gives  $y=-({a_{01}+a_{11}x})/\left({2(a_{02}+x^2)}\right)$, and it follows that 
$$
\frac{\partial p}{\partial x}(x,y)=\frac{1}{2(a_{02}+x^2)^2}q_5(x),
$$
where $q_5=b_0+\cdots +b_5x^5$, with $b_5=4a_{20}$. 

If $a_{20}\neq 0$ and $a_{02}=1$, we get a zero of $\nabla p$. If $a_{02}=-1$, we observe that $q_5(1)=(a_{01}+a_{11})^2$ and $q_5(-1)=-(a_{01}-a_{11})^2$. If $q_5(1)$ and $q_5(-1)$ are not zero, we have a zero of $q_5$ different from $\pm 1$, and thus a zero of $\nabla p$.
If $q_5(1)=0$, i.e. $a_{01}=-a_{11}$, we have 
\begin{equation}\label{8709}
\frac{\partial p}{\partial y}=(x-1)\left (a_{11}+2y+2xy\right).
\end{equation}
Then if $a_{11}=0$ it is simple to see that $\nabla p\left(-a_{10}/(2a_{20}),0\right)=(0,0)$. On the other hand, if $a_{11}\neq 0$,
\eqref{8709} gives that $x=\big(-(a_{11}+2y)/(2y)\big)$ annihilates $\frac{\partial p}{\partial y}$, and 
$$
\frac{\partial p}{\partial x}(x,y)=-\frac{y^3+(2a_{20}+a_{10})y+a_{20}a_{11}}{y},
$$
which clearly has a zero $y\neq 0$. Analogously, if $q_5(-1)=0$, we obtain a zero of $\nabla p$.

Thus we can suppose $a_{20}=0$. Solving $x$ in $\frac{\partial p}{\partial x}=0$, we get $x=-(a_{10}+a_{11}y)/\left(2y^2\right)$, which substituting in $\frac{\partial p}{\partial y}$ gives 
$$
\frac{\partial p}{\partial y}(x,y)=\frac{1}{2y^3}\left(a_{10}^2+a_{10}a_{11}y+2a_{01}y^3+4a_{02}y^4 \right).
$$
If $a_{10}\neq 0$ and $a_{02}=-1$, we have a zero of $\nabla p$. 
If $a_{10}\neq 0$ and $a_{02}=1$, we consider the discriminant of $p(x,y)-c=a_{10}x-c+(a_{01}+a_{11}x)y+(1+x^2)y^2$ as in Proposition  \ref{Level_Sets_2ttt}, and notice that it is a polynomial of degree $3$ in $x$. 
Therefore,  if $p$ is a submersion, all its level sets are connected. 

Now to finish this case, we suppose $a_{10}=0$ and observe that $\frac{\partial p}{\partial x}=(2xy+a_{11})y$. Then if $a_{11}\neq0$, we have that $\nabla p\left(-a_{01}/a_{11},0\right)=(0,0)$. On the other hand if $a_{11}=0$, we see that $\nabla p\left(0,-a_{01}/(2a_{02})\right)=(0,0)$ as well.

Finally, for case \eqref{IV_4}, we first observe that if $p$ is a submersion, then changing $x$ to $y$ if necessary, we can suppose $a_{01}\neq 0$.  
Then $\frac{\partial p}{\partial x}=0$ gives that $x=-\big(a_{10}+a_{11}y\big)/\big(2y^2\big)$ and hence $\frac{\partial p}{\partial y}=\big({2a_{01}y^3+a_{11}a_{10}y+
a_{10}+a_{10}^2}\big)/({2y^3})$.
Thus if $a_{10}\neq 0$, we get that $\nabla p$ has a zero. On the other hand,  if $a_{10}=0$, we get that $y=0$ annihilates $\frac{\partial p}{\partial x}$ and if $a_{11}\neq 0$, we have a zero of $\nabla p(x,0)$. 
Finally, if $a_{11}=0$, the change of variables $(x,y)\mapsto \big(x/\sqrt{|a_{01}|},a_{01}y \big)$ gives the second case of this proposition. 
\end{proof}

\subsection{Cases (IV) and (V)}\label{ultima}

\begin{proposition}\label{Pro_case_II_VI}
If the polynomials of cases (IV) and (V) are submersions, then all its level sets are connected.	
\end{proposition}

\begin{lemma}
The polynomials of cases (IV) and (V) are equivalent to one of the following 
\begin{align}
p=&a_{10}x+a_{01}y+a_{20}x^2+a_{11}xy+a_{02}y^2+a_{30}x^3+y^3+
x^4+a_{22}x^2y^2,\label{II_1}\\
p=&a_{10}x+a_{01}y+a_{20}x^2+a_{11}xy+a_{02}y^2+a_{30}x^3+x^4+
a_{22}x^2y^2,\label{II_2}
\end{align} 
with $a_{22}=\pm 1$.
\end{lemma}
\begin{proof}
Dividing polynomials of cases (IV) and (V) by $\alpha$ and $-\alpha$, respectively, and applying the transformation $T(x,y)=\big(y,\sqrt{6}x\big)$, we obtain they are equivalent to  
$$
a_{10}x+a_{01}y+a_{20}x^2\hspace{-.05cm}+a_{11}xy+a_{02}y^2\hspace{-.05cm}+a_{30}x^3\hspace{-.05cm}+a_{21}x^2y+
a_{12}xy^2\hspace{-.05cm}+a_{03}y^3\hspace{-.05cm}+x^4\hspace{-.05cm}+a_{22}x^2y^2, 
$$
with $a_{22}=\pm 1$. 
Then we compose this with the transformation $T(x,y)=\big(x+a_{12}/(2a_{22}),y+a_{21}/(2a_{22})\big)$ to eliminate $a_{12}$ and $a_{21}$. 
If $a_{03}\neq 0$, we compose the result with $T(x,y)=\big(x/a_{03},y/a_{03}\big)$ and by multiplying it by $1/a_{03}^{4}$ we obtain case \eqref{II_1}. 
On the other hand, if $a_{03}= 0$, we get case \eqref{II_2}. 
\end{proof}

\begin{proof}[Proof of Proposition \ref{Pro_case_II_VI}]
Take case \eqref{II_1} of the lemma. Observe that for each $c\in\R$, $p(x,y)-c$ has the form $y^3+A(x)y^2+B(x)y+C(x)$ as in 
Proposition \ref{Level_Sets_3}. Calculating the discriminant $D(x)$ we observe it is a polynomial of degree $10$ with leader coefficient $a_{22}^3/27$.
Thus if $a_{22}=1$, we have that if $p$ is a submersion,  all its level sets are connected. 

In case $a_{22}=-1$, we observe that $\frac{\partial p}{\partial y}$ and $\frac{\partial p}{\partial x}$ have the forms $M(y)x^2+a(y)x+b(y)$ and $N(y)x^3+c(y)x^2+d(y)x+e(y)$, respectively, with $M(y)=-2y$ and $N(y)=4$. 
With the notations of Theorem \ref{raiz_23}, 
$$
R_0\left(\frac{\partial p}{\partial y},\frac{\partial p}{\partial x},x\right)=b_0+b_1y+\cdots+b_8y^8,
$$
with $b_8=-48$, and 
$$
R_1\left(\frac{\partial p}{\partial y},\frac{\partial p}{\partial x},x\right)=4a_{11}^2+(8a_{01}+6a_{30}a_{11})y+
(16a_{02}+8a_{20})y^2+24y^3-8y^4.
$$
If $a_{11}\neq 0$, there exist $z>0$ such that $R_1(z)=0$. 
Thus $b_8M(z)=96z>0$, and Corollary \ref{raiz_23_Cor}, gives us $p$ is not a submersion. 
Now if $a_{11}=0$, we observe that $b_0=-16a_{01}^3$. If $a_{01}<0$, there exists $z\in\R$ such that $b_8R_0(z)<0$, which by Theorem \ref{raiz_23},
guarantees $p$ is not a submersion. On the other hand, if $a_{01}>0$, there exists $z>0$ such that $R_1(z)=0$ and $b_8M(z)>0$, which again by Corollary \ref{raiz_23_Cor} guarantees that $p$ is not a submersion. 
Finally, if $a_{01}=0$, we see that $\frac{\partial p}{\partial y}(x,0)\equiv 0$ and $\frac{\partial p}{\partial x}(x,0)$ is a polynomial of degree $3$, which always has a zero.

Now in the case  \eqref{II_2}, $\frac{\partial p}{\partial y}$ and $\frac{\partial p}{\partial x}$ also have their  forms as in Theorem \ref{raiz_23}, with $M(y)=2a_{22}y$ and $N(y)=4$. 
Here $R_0(y)=b_0+\cdots+b_6y^6+b_7y^7$, 
with $b_7=-32a_{22}^4a_{02}$. Thus by Theorem \ref{raiz_23}, $p$ is not a submersion if $a_{02}\neq0$.

On the other hand, if $a_{02}=0$, we divide the analysis in two cases: 
$a_{22}=1$ and $a_{22}=-1$. In the first case, if $a_{01}=0$, $y=-a_{11}/\left(2x\right)$ annihilates $\frac{\partial p}{\partial y}$ and transforms   $\frac{\partial p}{\partial x}=0$ in an equation of degree $3$ in $x$ with constant term $a_{10}\neq 0$ (otherwise $\nabla p(0,0)=(0,0)$). Thus we get a zero of $\nabla p$. 
If $a_{01}\neq 0$, we observe that  $p(x,y)-c$ has the form 
$$
A(x)y^2+B(x)y+C(x),
$$
with $A(x)=x^2$, $B(x)=a_{01}+a_{11}x$ and $C(x)=-c+a_{10}x+a_{20}x^2+a_{30}x^3+x^4$. 
If $x=0$, there is exactly one $y$ such that $p(0,y)=0$. Moreover, the discriminant $\Delta(x)$ is a polynomial of degree $6$ with leader coefficient $-4$. 
Therefore, we are under the hypotheses of Proposition \ref{Level_Sets_22tt}, and hence if $p$ is a submersion, all its level sets are connected. 

Finally, in the second case, i.e. $a_{22}=-1$, we observe that $\frac{\partial p}{\partial y}$ and $\frac{\partial p}{\partial x}$ are as in Theorem \ref{raiz_23}, with $M(y)=-2y$ and $N(y)=4$, and 
$$
R_0\left(\frac{\partial p}{\partial y},\frac{\partial p}{\partial y},x\right)=b_0+b_1y+\cdots +b_6 y^6,
$$
with $b_6=-16a_{01}$, and 
$$
R_1\left(\frac{\partial p}{\partial y},\frac{\partial p}{\partial x},x\right)=4a_{11}^2+\left(8a_{01}+6a_{11}a_{30}\right)y+8a_{20}y^2-8y^4.
$$
Then, if $a_{01}\neq 0$ and $a_{11}\neq 0$, we have by Corollary \ref{raiz_23_Cor} that $p$ is not a submersion. If $a_{01}\neq 0$ and $a_{11}=0$, we also get by this corollary that $p$ is not a submersion.
On the other hand, if $a_{01}=a_{11}=0$, it is simple to conclude that $\nabla p$ has a zero. 
\end{proof}

\section{The polynomials of Theorem \ref{Classification}}\label{quase_final}
Let $M\subset \R^2$ be an open set and $f:M\to\R$ be a smooth submersion. We denote by $\mathscr{F}(f)$ the foliation of $M$ given by the connected components of the level sets of $f$ (Lemma \ref{unbounded}). 
We recall the concept of half-Reeb component.
\begin{definition}
Let  $f:\R^2\to\R$ be a smooth submersion, $h_0:\R^2\setminus\{0\}\to\R$ be defined by $h_0(x,y)=xy$, and  
$$
B=\left\{(x,y)\in[0,2]\times[0,2]\ |\ 0<x+y\leq 2\right\}.
$$ 
We say that $\mathcal{A}\subset \R^2$ is a \emph{half-Reeb component}, or simply a \emph{hRc}, of $\mathscr{F}(f)$ if there is a homeomorphism $T:B\to\mathcal{A}$ which is a topological equivalence between $\mathscr{F}(h_0)|_{B}$ and $\mathscr{F}(f)|_{\mathcal{A}}$ with the following properties:
\begin{enumerate}
\item The segment $\{(x,y)\in B\ |\ x+y=2\}$ is sent by $T$ onto a transversal section to the leaves of $\mathscr{F}(f)$ in the complement of $T(1,1)$. This section is called the \emph{compact edge} of $\mathcal{A}$.
\item Both the segments $\{(x,y)\in B\ |\ x=0\}$ and $\{(x,y)\in B\ |\ y=0\}$ are sent by $T$ onto full half leaves of $\mathscr{F}(f)$, called the \emph{non-compact edges} of $\mathcal{A}$.
\end{enumerate} 
\end{definition}

\begin{figure}[!htpb]
\begin{center}
%\psfrag{T}{$T$}
%\psfrag{B}{{$B$}}
%\psfrag{A}{{$\mathcal{A}$}}
\includegraphics[scale=.55]{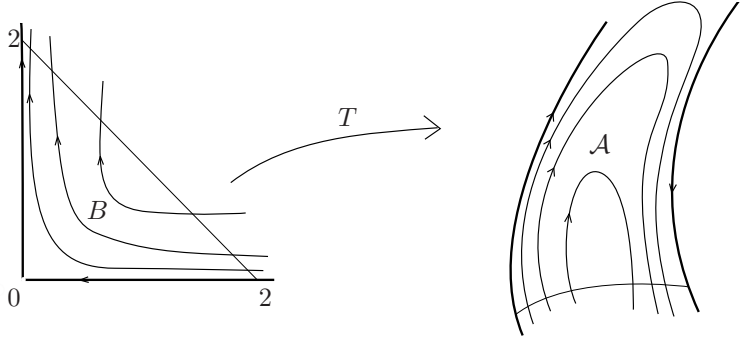}
\begin{picture}(10,20)
\put(-295,12){$0$}
\put(-200,12){$2$}
\put(-295,110){$2$}
\put(-265,45){$B$}
\put(-75,70){$\mathcal{A}$}
\put(-170,80){$T$}
\end{picture}
\caption{Definition of hRc}\label{HalfReeb}
\end{center}
\end{figure}

The existence of hRc is equivalent to the existence of inseparable leaves on the foliation $\mathscr{F}(f)$ (see \cite{GJLT} for details). 
Moreover, the following proposition of \cite{BS} relates this definition to connectedness of level sets. 

\begin{proposition}
Let  $f:\R^2\to\R$ be a smooth submersion. Then $\mathscr{F}(f)$ has a hRc if and only if there exists $c\in\R$ such that $f^{-1}\{c\}$ is not connected. 
\end{proposition}

A particular version of the following result is already  contained in \cite{BS}.  The proof in the general case is a simple extension of that and we add it here for completeness.

\begin{proposition}\label{integral}
Let $f:\R^2\to\R$ be a smooth submersion, $\mathcal{A}\subset\R^2$ be a hRc of $\mathscr{F}(f)$ and $U$ be a neighbourhood of $\mathcal{A}$. 
If $h:U\to[0,\infty)$ is a measurable function such that 
$$
\int_{\mathcal{A}}h=\infty,
$$
then there is not a differentiable $g:U\to\R$ such that $\det D(f,g)=h$ at $U$.
\end{proposition}

\begin{proof}
Let $\gamma:[0,1]\to\mathcal{A}$ be an injective curve that parametrize the compact edge of $\mathcal{A}$. 
For $n$ big enough, the leaf of $\mathscr{F}(f)$ through $\gamma(1/n)$ cuts $\gamma$ in $\gamma(t_n)$, with $t_n>1/n$. 
We denote by $\beta_n$ the interval of this leaf between $\gamma(1/n)$ and $\gamma(t_n)$, and by $\gamma_n$ the interval of $\gamma$ between this two points. 
We denote yet by $B_n$ the compact region bounded by $\gamma_n$ and $\beta_n$. By the Monotone Convergence Theorem, 
\begin{equation}\label{89077}
\int_{\mathcal{A}}h=\lim_{n\to\infty} \int_{B_n}h.
\end{equation}
If there is $g$ such that $h=\det D(f,g)$ in $U$,  Green's Theorem gives 
$$
\int_{B_n}h=\int_{\gamma_n}-g\left(\frac{\partial f}{\partial x}dx+\frac{\partial f}{\partial y}dy\right) \leq \int_{\gamma}-g\left(\frac{\partial f}{\partial x}dx+\frac{\partial f}{\partial y}dy\right) ,
$$
since $\beta_n$ is  orthogonal to $\nabla f$. Thus $\int_{B_n}h$ is uniformly bounded, a contradiction with \eqref{89077} and $\int_{\mathcal{A}}h=\infty$.
\end{proof}

We shall also need the following simple lemma (a reduced version of it was already   used  in  \cite{BS}). 
\begin{lemma}\label{Calculus}
Let $b_1,b_2>0$, and $\phi_1:(0,b_1)\to\R$ and $\phi_2:(b_2,\infty)\to\R$ be defined by $\phi_1(t)=\phi_2(t)=\sum_{j=k_1}^{k_2} c_j t^j$, where $k_1\leq k_2\in\Z$ and $c_{k_1}, \, c_{k_2}\neq 0$. If $\phi_i(t)>0$, for all $t$ in the domain of $\phi_i$, then
\begin{enumerate}
\item\label{1045} $c_{k_i}>0$, for $i=1,2$. 
\item\label{1046} If $k_1\leq -1$, then $\int_0^{b_1}\phi_1(t)dt=\infty$.
\item\label{1047} If $k_2 \geq -1$, then $\int_{b_2}^{\infty}\phi_2(t)dt=\infty$.
\end{enumerate}
\end{lemma}
\begin{proof}
If we multiply $\phi_i(t)$ by $t^{-k_i}$, hypothesis in \ref{1045} gives  that $\sum_{j=k_1}^{k_2}c_{j}t^{j-k_i}>0$, for all $t$ in the domain of $\phi_i$. 
Taking $t\to 0$, if $i=1$ or $t\to\infty$, if $i=2$, it follows that $c_{k_i}>0$, proving \ref{1045}. The proof of \ref{1046} and \ref{1047} follows by H\"older's inequality: defining $I_1=(0,b_1)$ and $I_2=(b_2,\infty)$, we get that $b_i^{-1-k_i}\int_{I_i}\phi_i(t)dt \geq \int_{I_i}t^{-1-k_i}\phi_i(t)dt=\infty$, since for both $i=1,2$, the last integral is $\int_{I_i}c_{k_i}t^{-1}dt$ plus the integral of a $L^1(I_i)$ function.    
\end{proof}

We now apply these results to analyse each polynomial of Theorem  \ref{Classification} in each of the subsections bellow. 
As a consequence, we will obtain the following theorem.

\begin{theorem}\label{naoq}
If $p(x,y)$ is one of the polynomials of Theorem \ref{Classification}, then there is not a polynomial $q(x,y)$ such that $\det D(p,q)(x,y)>0$, $\forall (x,y)\in\R^2$.  
\end{theorem}

\subsection{Case \ref{1}}
We consider the polynomial $p(x,y)=y+xy^2+y^4=y(1+xy+y^3)$. It is simple to observe that the following set is a hRc of $\mathscr{F}(p)$:
$$
\mathcal{A}=\left\{(x,y)\in\R^2\ |\ -1\leq y<0 \textnormal{ and } 0\leq x\leq -1/y-y^2 \right\}\cup\left\{(0,x)\ |\ x\geq 0\right\}.
$$ 
\emph{We claim that given a polynomial $h(x,y)=\sum_{i+j\leq k}b_{ij}x^iy^j$ such that $h(x,y)>0$, $\forall (x,y)\in\R^2$, then $\int_{\mathcal{A}}h=\infty$}. Thus by Proposition \ref{integral},  there is not a polynomial $q(x,y)$ such that $\det D(p,q)(x,y)>0$.

To prove the claim, we define
\begin{equation}\label{tau1}
\tau=\min\left\{j-i-1\ |\ 0\leq i+j\leq k \textnormal{ and } b_{ij}\neq 0\right\}.
\end{equation}
Applying the change of variables $(x,y)\mapsto \left(-xy/(1+y^3),-y\right)$ in the interior of $\mathcal{A}$, we obtain 
\begin{align}\label{Integral}
\int_{\mathcal{A}}h&=
\int_{0}^1\int_{0}^{1}\sum_{i+j\leq k}b_{ij}(-1)^jx^iy^{j-i-1}
\left(1-y^3\right)^{i+1}dxdy \notag\\
&=\int_{0}^1\int_{0}^1\big(s(x)y^{\tau}+s_1(x)y^{\tau +1}+\cdots\big)dydx,
\end{align}
where
$$
s(x)=\sum_{\tatop{i+j\leq k}{j-i-1=\tau}}b_{ij}(-1)^{j}x^i, 
$$
and $s_1(x),s_2(x),\ldots$ are suitable polynomials in $x$.

By \eqref{tau1}, the polynomial  $s(x)$ is not identically zero, hence  there exists a $0<c\leq 1$ such that $s(x)\neq 0$, $\forall x\in(0,c)$. Since $h(x,y)>0$, for all $(x,y)\in\R^2$, it follows by Lemma \ref{Calculus} that 
$s(x)>0$ for all $x\in(0,c)$. Moreover, as $b_{00}\neq0$, we have that $\tau<0$ and thus, by applying Lemma \ref{Calculus} once more, it follows that, for each $x\in (0,c)$, 
$$
\int_{0}^{1}\left(s(x)y^{\tau}+s_1(x)y^{\tau+1}+\cdots\right)dy=\infty.
$$
Then $$
\int_0^c\int_0^1\left(s(x)y^{\tau}+s_1(x)y^{\tau+1}+\cdots\right)dydx=\infty,
$$
which by  \eqref{Integral} gives that $\int_{\mathcal{A}}h=\infty$, and the claim is proven.

\subsection{Case \ref{2}}
We consider  the quite analogous case of the polynomial $p(x,y)=y+a_{02}y^2+xy^3$, with $a_{02}=0$ or $1$. We observe that the closure of the following  set is a hRc of $\mathscr{F}(p)$.
$$
\mathcal{A}=\left\{(x,y)\in\R^2\ |\ -1\leq y<0 \textnormal{ and } -{1}/{y^2}-{a_{02}}/{y}\leq x
\leq a_{02}-1\right\}.
$$
\emph{We claim that given a positive polynomial $h(x,y)$, then $\int_{\mathcal{A}}h=\infty$}. As above, this will show there is not a polynomial $q$ such that $\det D(p,q)>0$.

The proof of this claim is similar to the one made above: by defining $\tau=\min\{j-2i-2\ |\ b_{ij}\neq 0\}$ and taking the bounded set $B=\{(x,y)\ |\ -1\leq y<0 \textnormal{ and } a_{02}-1\leq x\leq 0\}$, we get\footnote{Apply the change of variables $(x,y)\mapsto\left(-xy^2/(1+a_{02}y),-y\right)$.}   
\begin{align*}
\int_{\mathcal{A}\cup B}h &=\int_0^1\int_0^1\sum_{i+j\leq k}b_{ij}(-1)^{i+j}x^iy^{j-2i-2}
\left(1-a_{02}y\right)^{i+1}dydx. 
\end{align*}
Then since $\tau<0$, we get as before that $\int_{\mathcal{A}\cup B}h=\infty$, hence $\int_{\mathcal{A}}h=\infty$, proving the claim.

\subsection{Case \ref{3}}\label{33333}
We consider now the polynomial $p(x,y)=y+x^2y^2$. This is quite different from the former ones. We first observe that the following set  is a hRc of $\mathscr{F}(p)$.
\begin{equation*}\label{thirdcase}
\mathcal{A}=\left\{(x,y)\in\R^2\ |\ x\geq 1 \textnormal{ and } -1/x^2\leq y\leq 0\right\}.
\end{equation*}
Then $\int_{\mathcal{A}}1=-\int_1^{\infty}1/x^2dx<\infty$, and thus Proposition \ref{integral} can not be used to prove that there is not a polynomial $q$ such that $\det D(p,q)=1$, for example. 
But we will use this proposition to eliminate candidates to be the polynomial $q$. Then we will use a different argument to show that there is not such a polynomial.

Let us suppose there exist $q(x,y)=\sum_{i+j\leq k}b_{ij}x^i y^j$ such that 
$$
h(x,y)=\det D(p,q)(x,y)=\sum_{i+j\leq k}b_{ij}\left(2(j-i)x^{i+1}y^{j+1}-ix^{i-1}y^j\right)>0,
$$
for all $(x,y)\in\R^2$. 
We define 
$$
\tau=\max\left\{i-2j-3\ |\ 0\leq i+j\leq k \textnormal{ and }b_{ij}\neq 0\right\},
$$
and \emph{we claim that $\tau<-1$}. Indeed, calculating $\int_{\mathcal{A}}h$ by applying the change of variables $(x,y)\mapsto (x,-x^2y)$, we obtain  
\begin{align}\label{793035298760012}
\int_{\mathcal{A}}h &  = \int_0^1\int_1^{\infty} \sum_{i+j\leq k}b_{ij}(-1)^{j+1}\big(2(j-i)y+i\big)y^jx^{i-2j-3}dxdy \\ \notag 
 & =\int_0^1\int_1^{\infty}\left( s(y)x^{\tau}+s_1(y)x^{\tau-1}+\cdots\right)dxdy,
\end{align}
where 
$$
s(y)=\sum_{\tatop{i+j\leq k}{i-2j-3=\tau}}b_{ij}(-1)^{j+1}\left(2(j-i)y+i\right)y^j
$$
and $s_i(y)$ are suitable polynomials in $y$. If $\tau \geq -1$, we have that  $s(y)$ is not identically zero and thus  there exists  $c\leq 1$ such that $s(y)\neq 0$ in $(0,c)$. 
Moreover, since $h(x,y)$ is a positive polynomial, it follows by Lemma \ref{Calculus} that $s(y)>0$ in $(0,c)$. By the same lemma, it follows  that the above integral is infinite. 
But this contradicts Proposition  \ref{integral} (since we are supposing there exists $q$ such that $\det D(p,q)=h$). Hence the claim is proven.

Now since we are supposing $h(x,y)>0$, for all $(x,y)\in\R^2$, considering $y=\theta x^{-2}$ we have that 
$$
\sum_{i+j\leq k}b_{ij}\big(2(j-i)\theta-i\big)\theta^{j}x^{i-2j-1}>0
$$
for all $x,\theta\in\R$, with $x\neq 0$. Since $b_{ij}=0$ for $i-2j-1>0$ (as $\tau <-1$), we have that 
\begin{equation}\label{ldetheta}
L(\theta)=\sum_{\tatop{i+j\leq k}{i-2j-1=0}}b_{ij}\big(2(j-i)\theta-i\big)\theta^{j}\geq 0,
\end{equation}
for all $\theta\in\R$. Then by Lemma \ref{bruna} bellow, it follows that $L(\theta)$ is identically zero, which guarantees in particular that $b_{10}=0$. But this is a contradiction with $-b_{10}=h(0,0)>0$.
Hence we conclude there is not a polynomial $q(x,y)$ such that $\det D(p,q)(x,y)>0$, for all $(x,y)\in\R^2$.

\subsection{Case \ref{4}}
Finally, we consider $p(x,y)=y+a_{02}y^2+y^3+x^2y^2$, with $a_{02}^2<3$. We shall use the notations of Subsection \ref{33333}. We first observe that the closure of
$$
\mathcal{A}=\left\{(x,y)\in\R^2\ |\ -1\leq y<0 \textnormal{ and } \sqrt{2-a_{02}}\leq x\leq \sqrt{-1/y-a_{02}-y}\right\}
$$
is a half-Reeb component of $\mathscr{F}(p)$. 
Given $m<1$, since the inequality $m^2<1+a_{02}y+y^2$ is valid for all $y\in (b,0)$, for some $b\geq-1$, we have that  
\begin{equation}\label{desigualdade}
m\sqrt{-1/y}<\sqrt{-1/y-a_{02}-y},\ \ \forall y\in(b,0).
\end{equation}  
We suppose also that $b\geq -m^2$ and  define
$$
\widetilde{\mathcal{A}}=\left\{(x,y)\in\R^2\ |\ x\geq m/\sqrt{-b}\  \textnormal{ and } -m^2/x^2\leq y< 0\right\}.
$$
Now if  $(x,y)\in \widetilde{\mathcal{A}}$, then  $b\leq y<0$ and $m/\sqrt{-b}\leq x\leq m\sqrt{-1/y}$. 
Hence  by \eqref{desigualdade}, we have that $(x,y)\in \mathcal{A}\cup C$, where $C$ is the bounded set $C=\big\{-1\leq y<0 \textnormal{ and } 1\leq x\leq \sqrt{2-a_{02}}\big\}$, if $1 \leq \sqrt{2-a_{02}}$, or $C=\emptyset$. 
Thus $\int_{\mathcal{A}\cup C}h\geq \int_{\widetilde{\mathcal{A}}}h$, for all positive function $h$ defined in $\R^2$, and hence $\int_{\mathcal{A}}h=\infty$ if $\int_{\widetilde{\mathcal{A}}}h=\infty$. 

As in the preceding subsection, we suppose there is a polynomial $q(x,y)=\sum_{i+j\leq k}b_{ij}x^iy^j$ such that $h(x,y)=\det D(p,q)(x,y)>0$, $\forall (x,y)\in\R^2$. Then we define $\tau=\max\{i-2j-3\ |\ b_{ij}\neq 0\}$ and \emph{we claim that $\tau<-1$}. Indeed calculating $\int_{\widetilde{\mathcal{A}}}h$ by applying the change of variables $(x,y)\mapsto \left(x,-x^2y\right)$, we get that
$$
\int_{\widetilde{\mathcal{A}}}h=\int_{0}^{m^2}\int_{\frac{m}{\sqrt{-b}}}^{\infty}\sum_{i+j\leq k}b_{ij}(-1)^{j+1}\big(\left(2(j-i)y+i\right)x^{i-2j-3}y^j +r_{ij}(x,y)
\big)dxdy,
$$
where $r_{ij}(x,y)=3ix^{i-2j-7}y^{j+2}-2a_{02}ix^{i-2j-5}y^{j+1}$. As in Subsection \ref{33333}, it follows that if $\tau \geq -1$, this integral is infinite, which is a contradiction with Proposition \ref{integral}, and the claim is proven.

By considering again $y=\theta x^{-2}$, we have that 
$$
\sum_{i+j\leq k}b_{ij}\big((2(j-i)\theta-i)\theta^{j}x^{i-2j-1}-2a_{02}\theta^{j+1}x^{i-2j-3}
-3i\theta^{j+2}x^{i-2j-5}\big)>0,
$$
for all $\theta,x\in\R$, with $x\neq 0$. Since $\tau<-1$, it follows in particular that the polynomial $L(\theta)$ defined in \eqref{ldetheta} is such that $L(\theta)\geq 0$ for all $\theta\in\R$. Therefore, by Lemma
\ref{bruna} bellow, we get that $L(\theta)$ is identically zero,  hence $b_{10}=0$, which is a contradiction with $-b_{10}=h(0,0)>0$. 
Thus we conclude that there is not a polynomial $q$ such that $\det D(p,q)>0$.

\begin{lemma}\label{bruna}
Given $N\in\N$ and $b_0, \dots, b_N\in\R$, let $L(\theta)$ be the polynomial 
\begin{equation*}
L(\theta)=\sum_{j=0}^{N}b_j\big(2(j+1)\theta +2j+1\big)\theta^j.
\end{equation*}
Then $L(\theta)$ is the zero polynomial or there exist $\theta_1,\theta_2\in \R$ such that $L(\theta_1)<0<L(\theta_2)$.
\end{lemma}
See the appendix for the proof of Lemma \ref{bruna}.

\section{Proof of Theorem \ref{main}}\label{fim}
Suppose we are under the hypotheses of Theorem \ref{main}. If all the level sets of $p$ are connected or the degree of $p$ is less than or equal to $3$, then $F$ is injective by the argument given in introduction or by the main result of \cite{BS}, respectively. 
Thus we can suppose that $p$ has least one disconnected level set and that the degree of $p$ is $4$. 

By Theorem \ref{Classification}, there exist $M,N\in\R$, $M\neq 0$, and an affine change of coordinates $T$ such that $\overline{p}=Mp\circ T^{-1} +N$ is one of the polynomials of this theorem. 
If we take $M'=\signal\left(\det DF\right)\det T^{-1}M$ and  $\overline{q}=M'p\circ T^{-1}$, we have that $\det D(\overline{p},\overline{q})=MM'\det DF \det T^{-1}>0$,  a contradiction with Theorem \ref{naoq}.

\appendix

\section{Proof of Lemma \ref{bruna}}

From now on, we consider the following  convention: given $P:\Z\to\R$,  
$$
\prod_{l=m}^n P(l)=1,\ \ \ \ \ \  \  \sum_{l=m}^nP(l)=0,
$$
if $n<m$. We also denote by $\det A$ the determinant of a quadratic matrix $A$.

\begin{lemma}\label{l1} Let $k\geq j\geq 1$ and $i\geq -1$ be integers. Define
\begin{equation*}
\alpha_i^k=(-2)^{i+1}\prod_{l=0}^i\frac{2k-l}{4k-(2l+1)},
\end{equation*}
and the $(j+1)\times(j+1)$ matrix
\begin{equation*}
H_j^k=\left(\begin{matrix}
\alpha_{2j-1}^k & \alpha_{2j-2}^k & \cdots & \alpha_j^k & \alpha_{j-1}^k\\
\alpha_{2j-2}^k & \alpha_{2j-3}^k & \cdots & \alpha_{j-1}^k & \alpha_{j-2}^k\\
\vdots & \vdots & \ddots & \vdots & \vdots\\
\alpha_j^k & \alpha_{j-1}^k & \cdots & \alpha_1^k & \alpha_0^k\\
\alpha_{j-1}^k & \alpha_{j-2}^k & \cdots & \alpha_0^k & 1
\end{matrix}\right).
\end{equation*}
Then the determinant of $H_j^k$ is positive.
\end{lemma}

\begin{proof}
For each integer  $l=0,\ldots,2j-1$, denote
$$
m_l=2k-l,\ \ \ \ \ \ \ n_l=4k-(2l+1).
$$
Denote also $\left(h_{r,s}\right)_{m\times n}$ the $m\times n$ matrix whose element in the $r$th row and $s$th column is $h_{r,s}$. 
Then $H_j^k=\left(h_{r,s}\right)_{(j+1)\times (j+1)}$, where 
\begin{align*}
h_{r,s} &=\alpha_{2j-r-s+1}^k
=(-2)^{2j-r-s+2}\prod_{l=0}^{2j-r-s+1}\frac{m_l}{n_l}\\
&=(-2)^{j-r+1}(-2)^{j-s+1}\prod_{l=0}^{j-s}\frac{m_l}{n_l}
\prod_{l=j-s+1}^{2j-r-s+1}m_l\prod_{l=j-s+1}^{2j-s}\frac{1}{n_l}\prod_{l=2j-r-s+2}^{2j-s}n_l.
\end{align*}
For each $r$, divide the row $r$ by $(-2)^{j-r+1}$ and for each $s$, divide the column $s$ by the factors above depending only on $s$. After that, we obtain that $\det H_j^k=M\det (a_{r,s})_{(j+1)\times (j+1)}$, where
$$
M=\prod_{r=1}^{j+1}(-2)^{j-r+1} \prod_{s=1}^{j+1}\left( (-2)^{j-s+1}\prod_{l=0}^{j-s}\frac{m_l}{n_l}\prod_{l=j-s+1}^{2j-s}\frac{1}{n_l}\right)
$$
and
$$
a_{r,s}=\prod_{l=j-s+1}^{2j-r-s+1}m_l\prod_{l=2j-r-s+2}^{2j-s}n_l. 
$$
Since $1\leq j\leq k$, we have that $m_l,n_l>0$ and it follows that $M$ is positive. Thus it is enough to prove that $\det (a_{r,s})_{(j+1)\times (j+1)}$ is positive. 

To prove this we will make operations with the rows and the columns of the matrix.

We first operate the rows. 
We denote $a^0_{r,s}=a_{r,s}$ for each $r,s$. 
Then for each $t=1,\dots, j$, we consider $b=t+1,\ldots,j+1$ and replace the row $b$ by this row minus two times the row $b-1$ and denote by $a_{r,s}^t$ the elements of this new matrix. That is, for each $t$,
\begin{equation*}
a_{r,s}^t=\left\{\begin{array}{ll} 
a_{r,s}^{t-1}-2a^{t-1}_{r-1,s}, & r>t\\
a_{r,s}^{t-1}, & r\leq t. 
\end{array}\right.
\end{equation*}
It is not difficult to prove by induction on $t$ that
$$
a_{r,s}^t=\prod_{l=0}^{t-1}(-2l-1)\prod_{l=j-s+1}^{2j-s-r+1}m_l\prod_{l=2j-s-r+(t+2)}^{2j-s}n_l,\ \ \ \ \ r>t.
$$
In the last step $t=j$, we obtain the matrix $\left(a_{r,s}^j\right)_{(j+1)\times(j+1)}$, with
$$
a_{r,s}^j=a_{r,s}^{r-1}=\prod_{l=0}^{r-2}(-2l-1)\prod_{l=j-s+1}^{2j-s-r+1}m_l
$$
since the products of the $n_l$ are equal to $1$ when $t=r-1$.

We now operate the columns of this modified  matrix. Denoting again $a_{r,s}=a_{r,s}^{j}$, for each $t=1,\dots,j$ and for each $b=1,\ldots,j+1-t$, replace the column $b$ by the column $b$ minus the column $b+1$. Denoting by $a_{r,s}^t$ the elements of this new matrix, we have
\begin{equation*}
a_{r,s}^t= \left\{\begin{array}{ll}
a_{r,s}^{t-1}-a_{r,s+1}^{t-1}, & s\leq j+1-t\\
a_{r,s}^{t-1}, & s>j+1-t.
\end{array}\right.
\end{equation*} 
It is not difficult to show by induction on $t$ that
$$
a_{r,s}^t=\prod_{l=1}^{t}(r-j+(l-2))\prod_{l=0}^{r-2}(-2l-1)\prod_{l=j-s+1}^{2j-s-r-(t-1)}m_l,\ \ \ \ \ \ s\leq j+1-t.
$$
In the last step $t=j$, we have the matrix $\left(a_{r,s}^j\right)_{(j+1)\times (j+1)}$, with
$$
a_{r,s}^j=a_{r,s}^{j+1-s}=\prod_{l=1}^{j+1-s}(l-(j-r+1)-1)
\prod_{l=0}^{r-2}(-2l-1)\prod_{l=j+1-s}^{j-r}m_l.
$$
If $r>s$, $l=j-r+2$ will annihilate the product on the first factor of $a_{r,s}^j$ and hence $\left(a_{r,s}^j\right)_{(j+1)\times(j+1)}$ is a triangular matrix. Moreover, if $s=r$, we have 
$$
a_{r,r}^j=\prod_{l=1}^{j+1-r}\left(l-(j+1-r)-1\right) 
\prod_{l=0}^{r-2}(-2l-1),
$$
which is the product of $j$ \emph{negative} factors. Thus $\prod_{r=1}^{j+1}a_{r,r}^j$ is the product of $j(j+1)$ negative factors.
Hence $\det H_j^k$ is positive.
\end{proof}

\begin{corollary}\label{c1}
The matrix of Lemma \ref{l1} has all its leading principal minors positive.
\end{corollary}

\begin{proof}
Given $i\geq 1$ and $k\geq 2$, it is simple to see that 
$$
\alpha_i^k=\alpha_1^k\alpha_{i-2}^{k-1}.
$$
Since $\alpha_1^r>0$, for all $r\geq 1$,  we have that the leading principal minors of $H_j^k$ are positive multiples of  $\det H_{j'}^{k'}$, with $j'\leq k'$, which are positive by Lemma \ref{l1}. 

\end{proof}

\begin{proof}[Proof of Lemma \ref{bruna}]
If $L(\theta)$ is   a positive polynomial, i.e. $L(\theta)\geq 0$, for all $\theta\in\R$, then the degree of $L$ is even, say $N=2k-1$. Besides it is well known that there are polynomials
$$
g(\theta)=\sum_{j=0}^k a_j\theta^j,\ \ \ \ \ \ \  h(\theta)=\sum_{j=0}^kc_j\theta^j
$$
such that
\begin{equation}\label{90000}
L=g^2+h^2.
\end{equation}
That is,
\begin{equation}\label{e1}
b_0+\sum_{j=1}^{2k-1}\Big(2jb_{j-1}+(2j+1)b_j\Big)\theta^j \hspace{-.06cm} + \hspace{-.06cm} 4kb_{2k-1}
\theta^{2k}=\sum_{j=0}^{2k}\left(\sum_{r+s=j} \hspace{-.2cm} a_ra_s \hspace{-.1cm}+\hspace{-.1cm}c_rc_s\hspace{-.1cm}\right)\theta^j .
\end{equation}

It follows from \eqref{e1} that 
\begin{equation}\label{eq3}
b_0=a_0^2+c_0^2, 
\end{equation}
\begin{equation}\label{eq4}
2jb_{j-1} +(2j+1)b_j=\sum_{r+s=j}a_ra_s+c_rc_s,  \ \ \ \ j=1,..,2k-1, 
\end{equation}
and
\begin{equation}\label{eq5}
4kb_{2k-1}=a_k^2+c_k^2.
\end{equation}

\emph{We assert that if we use only \eqref{eq3} and \eqref{eq4}, then  for each $j=0,\dots,2k-1$,
\begin{equation}\label{e2}
b_j=\frac{1}{j+1}\sum_{l=0}^j(-2)^l\prod_{i=0}^l\frac{j-(i-1)}{2(j-i)+1}\sum_{r+s=j-l}a_ra_s+c_rc_s.
\end{equation}}
Indeed, \eqref{eq3} shows \eqref{e2} for $j=0$. We assume \ref{e2} is true for $j$. 
Using \eqref{eq4} to write $b_{j+1}$ in terms of $b_j$, after elementary and long calculations we can prove that \eqref{e2} is true for $j+1$,
and the assertion is proven.

Denoting $a=(a_0,a_1,\ldots, a_k)$ and $c=(c_0,c_1, \ldots, c_k)$, we define $K(a,c)=-4kb_{2k-1}+a_k^2+c_k^2$. 
By \eqref{eq5}, $K(a,c)=0$, whereas by  \eqref{e2},
\begin{equation*}
K(a,c)=\sum_{l=0}^{2k-1}(-2)^{l+1}\prod_{i=0}^l\frac{2k-i}{4k-(2i+1)}\sum_{r+s=2k-1-l}(a_ra_s+c_rc_s)+a_k^2+c_k^2.
\end{equation*}

Now considering $(a,c)$ as variables in $\R^{2(k+1)}$, last equation shows that  $K(a,c)$ is a quadratic form in $\R^{2(k+1)}$. \emph{We assert that $K(a,c)$ is a positive definite quadratic form}. 
This guarantees that $(a,c)=(0,0)$ provided that $K(a,c)=0$. Thus $g$ and $h$ are identically zero, and by \eqref{90000}, $L$ is identically zero. 
Therefore if $L(\theta)\neq 0$, we have shown that there exists $\theta_1$ such that $L(\theta_1)<0$. 
To conclude the proof, we just apply this result to the polynomial 
$-L(\theta)$.

In order to show the assertion, we first observe that it is enough to prove that $K(a,0)$ is positive definite, since $K(a,c)=K(a,0)+K(0,c)$ and $K(a,0)=K(0,a)$. 
Now to verify that $K(a,0)$ is positive definite, we shall show that  $\frac{1}{2}\Hess(K)$, where $\Hess(K)=\left(\partial^2 K/\partial a_i\partial a_j\right)_{(k+1)\times (k+1)}$ is the Hessian matrix of $K$, has all its leading principal minors positive. 
It is simple to show that for each $r,s\in\{0,1,\ldots, k\}$, 
%\footnote{Observe that $\partial^2 K/\partial{a_k}^2$ is also counted here.} 
$$
\frac{1}{2}\frac{\partial^2K}{\partial a_r\partial a_s} =
 (-2)^{2k-r-s}\prod_{l=0}^{2k-r-s-1}\frac{2k-l}{4k-(2l+1)}.
$$
If we denote  
$$
\alpha_i^k=\frac{1}{2}\frac{\partial^2K}{\partial a_{r}\partial a_{s}},
$$
with $r+s=2k-i-1$, then $\frac{1}{2}Hess (K)=H_k^k$, where $H_k^k$ is the matrix of Lemma \ref{l1} with $j=k$. 
Thus by Corollary \ref{c1}, all the leading principal minors of $\frac{1}{2}Hess (K)$ are positive.

\end{proof}

%\begin{acknowledgements}
%\end{acknowledgements}



\end{document}